\documentclass[a4paper]{amsart}

\usepackage[english]{babel}
\usepackage[english]{babel}
\usepackage{amsmath,amssymb,amsthm}
\usepackage{amsfonts, amsmath, amssymb, amscd, amsthm}
\usepackage{multicol}
\usepackage{comment}
\newtheorem{theorem}{Theorem}[section]   
\newtheorem*{theorem*}{Theorem}          
\newtheorem{lemma}[theorem]{Lemma}
\newtheorem{proposition}[theorem]{Proposition}

\theoremstyle{definition}

\newtheorem{corollary}[theorem]{Corollary}
\newtheorem{example}[theorem]{Example}

\newtheorem{remark}[theorem]{Remark}

\newcommand{\metric}[2]{\ensuremath{\langle #1, #2\rangle}}  
\newcommand{\metri}[2]{\ensuremath{g( #1, #2)}}
\newcommand{\nks}{\ensuremath{\mathbb{S}^3 \times
\mathbb{S}^3}}   
\newcommand{\lcc}{\ensuremath{\tilde \nabla}}   
\renewcommand{\ne}{\ensuremath{\nabla^E}}   
\renewcommand{\epsilon}{\varepsilon}            
\newcommand{\e}{\epsilon}                       
\newcommand{\ijk}{\ensuremath{\epsilon_{ijk}}}  

\renewcommand{\H}{\ensuremath{\mathbb{H}}}

\newcommand{\Id}{\ensuremath{\mathrm{Id}}}

\numberwithin{equation}{section}

\DeclareMathOperator{\im}{Im}    

\title[Lagrangian submanifolds in the nearly K\"ahler $\mathbb{S}^3 \times
\mathbb{S}^3$]%
		{Lagrangian submanifolds in the nearly~K\"ahler~$\nks$}
\author{Bart Dioos, Luc Vrancken \and Xianfeng Wang}
\address{KU\ Leuven, Departement Wiskunde, Celestijnenlaan 200B,
  3001 Leuven, Belgium}
  \email{bart.dioos@wis.kuleuven.be}
\address{LAMAV, Universit\'e de Valenciennes, Campus du Mont Houy, 59313 Valenciennes Cedex 9, France; KU Leuven, Departement Wiskunde
Celestijnenlaan 200B, 3001 Leuven, Belgium} \email{Luc.Vrancken@univ-valenciennes.fr}
\address{School of Mathematical Sciences and LPMC,
Nankai University,
Tianjin 300071,  P. R. China}
\email{wangxianfeng@nankai.edu.cn}
\thanks {X. Wang was supported in part by NSFC (Grant Nos. 11201243 and 11571185) and ``Specialized Research Fund for the Doctoral Program of Higher Education, Grant No. 20120031120026''.}
\keywords {nearly K\"ahler manifold, Lagrangian submanifolds, constant sectional curvature, Lagrangian sphere,  Lagrangian torus.}
\subjclass[2010]{primary 53C42; secondary 53D12}

\begin{document}

\begin{abstract}
In this paper, we investigate Lagrangian submanifolds in the nearly K\"ahler $\nks$. We construct a new example which is a flat Lagrangian torus. We give a complete classification of all the Lagrangian immersions of spaces  of constant  sectional curvature in the nearly K\"ahler $\nks$. As a corollary,
we obtain that the radius of a round Lagrangian sphere in the nearly K\"ahler $\nks$  can only be $\frac{2}{\sqrt{3}}$ or $\frac{4}{\sqrt{3}}$.\end{abstract}

\maketitle

\section{Introduction}
The study of Lagrangian submanifolds originates from symplectic geometry and classical mechanics.
An even-dimensional manifold is called symplectic if it admits a symplectic form, which is a closed and non-degenerate two-form.
A submanifold of a symplectic manifold is called Lagrangian if the symplectic form restricted to the manifold vanishes and if the dimension of the submanifold is half the dimension of the symplectic manifold.
The well-known theorem of Darboux states that locally all symplectic manifolds are indistinguishable.
If one considers a Lagrangian submanifold immersed in a symplectic manifold, then by the theorem of Darboux
this Lagrangian submanifold can also be locally immersed in any symplectic manifold of the same dimension.
Therefore a local classification of Lagrangian submanifolds is trivial from the symplectic point of view.

Lagrangian submanifolds can more generally be considered in almost Hermitian manifolds.
Note that an almost Hermitian manifold is not necessarily symplectic.
We call that a submanifold of $M$ in an almost Hermitian manifold $N$ Lagrangian, if
the almost complex structure $J$ interchanges the tangent and the normal
spaces and if the dimension of $M$ is half the dimension of $N$.
The most important class of almost Hermitian manifolds are the K\"ahler manifolds.
K\"ahler manifolds admit a complex, Riemannian and symplectic structure which are all three compatible with each other.
The study of Lagrangian submanifolds in K\"ahler manifolds is a classic topic and was initiated in the 1970's~\cite{chenogiue}.
A classification of Lagrangian submanifolds from the Riemannian point of view is far from trivial.
There is no complete classification and this is too much to hope for.
For this reason it makes sense to study Lagrangian submanifolds with some additional Riemannian conditions.
For instance, one can study Lagrangian submanifolds that are minimal, Hamiltonian minimal, Hamiltonian stable or unstable
(see for instance \cite{MO2009},\cite{MO2014},\cite{oh})
or have constant sectional curvature~\cite{ejirilagr}.
For a review on Riemannian geometry of Lagrangian submanifolds we refer to~\cite{chenreview} and the references therein.

Nearly K\"ahler manifolds are almost Hermitian
manifolds with almost complex structure $J$ satisfying that
$\tilde\nabla J$ is skew-symmetric.
The geometry of nearly K\"ahler manifolds was initially studied by Gray \cite{Gr1, Gr3} in the 1970s from the point of view of weak holonomy.
Nagy (\cite{Nag1, Nag2}) made further contribution to the classification of nearly K\"ahler manifolds using previous work in \cite{CS2004}.
Butruille (\cite{butruille2, butruille}) proved that the only
homogeneous 6-dimensional nearly K\"ahler manifolds are the nearly K\"ahler $\mathbb{S}^6$,
$\mathbb{S}^3 \times \mathbb{S}^3$, the complex projective space
$\mathbb{CP}^3$ and the flag manifold $SU(3)/U(1)\times U(1)$.
In \cite{MSTAMS}, Moroianu and Semmelmann studied the infinitesimal Einstein deformations of nearly K\"ahler metrics.
Lagrangian submanifolds of the nearly K\"ahler $\mathbb{S}^6$ are well studied by now, see for instance \cite{ejiri}, \cite{DOVV},\cite{classification},\cite{DVTAMS} and \cite{Lotay}.
We also refer to Section~18 of~\cite{chenreview} and Chapter~19 of~\cite{chenpseudoriemannian} for an overview.
Moroianu and Semmelmann~\cite{ms} recently gave new examples of Lagrangian immersions of round spheres and Berger spheres in the nearly K\"ahler~$\nks$.
A broader study of Lagrangian submanifolds in nearly K\"ahler manifolds was investigated in~\cite{schafer} by Sch\"afer and Smozcyk.
It was proven in~\cite{schafer} that Lagrangian submanifolds in a nearly K\"ahler manifold behave nicely with respect to
the splitting of the nearly K\"ahler manifold. If a nearly K\"ahler manifold is a product of nearly K\"ahler manifolds,
then its Lagrangian submanifolds split into products of Lagrangian submanifolds.
Loosely speaking, this means that Lagrangian submanifolds in six-dimensional nearly K\"ahler manifolds are
building blocks of Lagrangian submanifolds in general nearly K\"ahler manifolds.
This motivates the study of Lagrangian submanifolds in six-dimensional nearly K\"ahler manifolds.
The existence for Lagrangian submanifolds in nearly K\"ahler manifolds is not unobstructed.
Sch\"afer and Smozcyk \cite{schafer} proved that Lagrangian submanifolds in a strict nearly K\"ahler manifold of dimension six
or a twistor space over a positive quaternionic K\"ahler manifold are minimal and orientable. This is different with Lagrangian submanifolds in K\"ahler manifolds.
The reason is that there is no Darboux theorem for nearly K\"ahler manifolds since these manifolds are not symplectic.
This is an extra reason to study these Lagrangian submanifolds from a Riemannian point of view.

In this paper, we study Lagrangian submanifolds  in the nearly K\"ahler~  $\nks$.
In Section~\ref{sec:lagrsubm}, we recall  the basic properties of the nearly K\"ahler $\nks$,
and present some properties of Lagrangian submanifolds in nearly K\"ahler manifolds.
In Section~\ref{sec:lsnks}, we show that on a  Lagrangian submanifold in the nearly K\"ahler~$\nks$ there exist
a local  frame and three angle functions that describe the geometry and shape of the submanifold very well.
These are related to the almost product structure $P$ introduced in \cite{bddv}.
We show that most of the geometry of the  submanifold~$M$ can be described in terms of the three angle functions. For example, the derivatives of these angle functions give information about most of the components of the second fundamental form.
In Section~\ref{sec:exlag}, we present eight examples (or families of examples) of Lagrangian submanifolds in the nearly K\"ahler~$\nks$.
The examples are Lagrangian immersions of respectively round spheres, Berger spheres  or a flat torus. The flat torus (see Example 4.8) is a new example.
Examples 4.1- 4.3 are the factors and the diagonal which were given by Sch\"afer and Smozcyk in \cite{schafer}. Examples 4.4-4.7 were constructed by Moroianu and Semmelmann in ~\cite{ms}.
In section \ref{sec:lscc}, we classify the Lagrangian submanifolds  of constant  sectional curvature in the nearly K\"ahler~$\nks$. The  main result that we prove is the following:
\begin{theorem}\label{main}
Let~$M$ be a Lagrangian submanifold  of constant  sectional curvature in the nearly K\"ahler~$\nks$.
Then up to an isometry of
the nearly K\"ahler~$\nks$, ~$M$ is locally congruent with  one of the following
immersions:
\begin{enumerate}
\item  $f\colon \mathbb{S}^3 \to\nks: u\mapsto (u,1)$, which is Example \ref{ex:l1},
\item $f\colon \mathbb{S}^3 \to\nks: u\mapsto (1,u)$, which is Example \ref{ex:l2},
\item $f\colon \mathbb{S}^3\to\nks: u\mapsto (u,u)$, which is Example \ref{ex:l3},
\item  $f\colon \mathbb{S}^3\to \nks: u\mapsto (u\mathbf{i}u^{-1},u\mathbf{j}u^{-1})$, which is Example \ref{ex:l6},
\item $f: \mathbb R^3\to \nks: (u,v,w)\mapsto (p(u,w),q(u,v))$, where $p$ and $q$ are constant mean curvature tori in $\mathbb{S}^3$ given in Example \ref{ex:l8}.
 \end{enumerate}
\end{theorem}

\begin{remark}
In view of  Propsition 4.4 in \cite{ms}, Moroianu and Semmelmann showed that the radius of a round Lagrangian sphere in the nearly K\"ahler~$\nks$ is necessarily of the form $\frac{k}{\sqrt{3}}$ (note that the scaling in \cite{ms} is slightly different from ours, so the radius here has been modified to be adapted to our conventions)
for some integer $k\geq 2$. As a corollary of our Theorem  \ref{main}, we obtain  that the values of the integer $k$ can only be $2$ or $4$.
\end{remark}

We remark that in \cite{ZHDVW}, the authors obtain the following complete classification of all the totally geodesic Lagrangian immersion in the nearly K\"ahler~$\nks$.
\begin{theorem}[\cite{ZHDVW}]\label{thmtg}
Let $M$ be a totally geodesic Lagrangian submanifold in the nearly K\"ahler~$\nks$. Then up to an isometry of
the nearly K\"ahler~$\nks$, ~$M$ is locally congruent with one of the following
immersions:
\begin{enumerate}
\item  $f\colon \mathbb{S}^3 \to\nks: u\mapsto (u,1)$, which is Example \ref{ex:l1},
\item $f\colon \mathbb{S}^3 \to\nks: u\mapsto (1,u)$, which is Example \ref{ex:l2},
\item $f\colon \mathbb{S}^3\to\nks: u\mapsto (u,u)$, which is Example \ref{ex:l3},
\item $f\colon \mathbb{S}^3\to\nks: u\mapsto (u,u\mathbf{i})$,  which is Example \ref{ex:l4},
\item $f\colon \mathbb{S}^3\to \nks: u\mapsto(u^{-1},u\mathbf{i}u^{-1})$,  which is Example \ref{ex:l5},
\item $f\colon \mathbb{S}^3\to \nks: u\mapsto(u\mathbf{i}u^{-1},u^{-1})$,  which is Example \ref{ex:l7}.
 \end{enumerate}
\end{theorem}
Hence, combing this result together with our main theorem in this paper, one obtain characterizations of all the eight examples (see Section 4 for details of the examples)
of Lagrangian submanifolds in the nearly K\"ahler~$\nks$.

\section{The nearly K\"ahler $\nks$ and its Lagrangian submanifolds}
\label{sec:lagrsubm}\label{sec:nks}

In this section we  recall the definition of the nearly K\"ahler $\nks$ from \cite{bddv} and
\cite{dlmv} and give some basic properties of  Lagrangian submanifolds which will be useful for the rest of the paper.

Using the natural identification $T_{(p,q)}(\nks) \cong T_p \mathbb{S}^3 \oplus T_q \mathbb{S}^3$, we
write a tangent vector at $(p,q)$ as  $Z(p,q) = \bigl( U(p,q), V(p,q)\bigr)$ or simply $Z=(U,V)$.

The $3$-sphere
$\mathbb{S}^3$ can be regarded as the set of all the unit quaternions in
$\mathbb{H}$, as usual we use the notations
$\mathbf{i},\,\mathbf{j},\,\mathbf{k}$ to denote the imaginary units of
$\mathbb{H}$.
Define the vector fields
\begin{align*}
 E_1(p,q) &= (p\mathbf{i},0),   &   F_1(p,q) &= (0,q\mathbf{i}),\\
 E_2(p,q) &= (p\mathbf{j},0),   &   F_2(p,q) &= (0,q\mathbf{j}),\\
 E_3(p,q) &= -(p\mathbf{k},0),  &   F_3(p,q) &= -(0,q\mathbf{k}).
\end{align*}

These vector fields are mutually orthogonal with respect to the   usual Euclidean product metric
on the nearly K\"ahler~$\nks$. The Lie brackets are $[E_i,E_j]=-2\ijk E_k$, $[F_i,F_j]=-2\ijk F_k$ and~$[E_i,F_j]=0$, where $$\varepsilon_{ijk}=\left\{
\begin{aligned}
& 1, \qquad\,  {\rm if}\ \{ijk\}\  {\rm is\ an\ even\ permutation\ of\ \{123\}},\\
&-1,\quad  {\rm if}\ \{ijk\}\ {\rm is\ an\ odd\ permutation\ of\ \{123\}},\\
& 0, \ \qquad {\rm \ otherwise}.
\end{aligned}
\right.
$$

The almost complex structure~$J$ on  the nearly K\"ahler~$\nks$ is defined by
\begin{equation}
\label{eq:acstructure}
   J(U,V)_{(p,q)} = \frac{1}{\sqrt{3}}\left( 2pq^{-1}V - U, -2qp^{-1}U + V \right)
\end{equation}
for $(U,V) \in T_{(p,q)}(\mathbb{S}^3\times \mathbb{S}^3)$ (see \cite{butruille}). Note that the definition uses
the Lie group structure of the nearly K\"ahler~$\nks$.
The map
\[
     T_{(1,1)}\nks\to T_{(1,1)}\nks : (U,V)\mapsto \frac{1}{\sqrt{3}}(2V - U, -2 V + U)
\]
defines an almost complex structure on the Lie algebra, the tangent space at~$(1,1)$.
By using left translations on the nearly K\"ahler~$\nks$ this map can be extended to an almost complex structure on the nearly K\"ahler~$\nks$.
The left translations on the nearly K\"ahler~$\nks$ are given by left multiplications with a unit quaternion.
The almost complex structure can be described as follows.
The first step is to left translate a vector~$(U,V)$ at~$(p,q)\in \nks$
to $(p^{-1}U, q^{-1}V)$ at the unit element~$(1,1)$. Then this vector is mapped onto
$\tfrac{1}{\sqrt{3}}(2 q^{-1}V-p^{-1}U, -2p^{-1}U + q^{-1}V)$ at the point~$(1,1)$.
When this vector is translated back to~$T_{(p,q)}\nks$, it gives the expression \eqref{eq:acstructure}.

The nearly K\"ahler metric on~$\nks$ is the Hermitian metric associated to the usual Euclidean product
metric on~$\nks$:
\begin{align*}
   g(Z,Z') &= \frac{1}{2} \left(\metric{Z}{Z'} + \metric{JZ}{JZ'}\right)\\
          &= \frac{4}{3} \left(\metric{U}{U'} +  \metric{V}{V'}\right)
             -\frac{2}{3} \left(\metric{p^{-1}U}{q^{-1}V'} +  \metric{p^{-1}U'}{q^{-1}V}\right),
\end{align*}
where~$Z=(U,V)$ and $Z'=(U',V')$. In the first line~$\metric{\cdot}{\cdot}$ stands for the usual Euclidean product metric on~$\mathbb{S}^3\times \mathbb{S}^3$ and in the second line $\metric{\cdot}{\cdot}$ stands for the usual Euclidean metric on~$\mathbb{S}^3$.
By definition the almost complex structure is compatible with the metric~$g$.
An easy calculation gives
\begin{align*}
g(E_i, E_j) &= 4/3  \,\delta_{ij}, &
g(E_i, F_j) &= -2/3 \,\delta_{ij}, &
g(F_i, F_j) &= 4/3  \,\delta_{ij}.
\end{align*}
Note that this metric differs up to a constant factor from the one introduced in \cite{butruille}.
Here we set everything up so that it equals the Hermitian metric associated with the usual Euclidean product metric.
In \cite{butruille}, the factor was chosen in such a way that the standard basis $E_1,E_2,E_3,F_1,F_2,F_3$ has volume 1.

\begin{lemma}[\cite{bddv}]
\label{lem:levicivita}
The Levi-Civita connection $\tilde \nabla$ on $\nks$ with respect to the metric~$g$ is given by
\begin{align*}
 \lcc_{E_i} E_j &= -\ijk E_k,                     &   \lcc_{E_i} F_j &= \frac{\ijk}{3}(E_k - F_k),\\
 \lcc_{F_i} E_j &= \frac{\ijk}{3} (F_k - E_k),    &   \lcc_{F_i} F_j &= -\ijk F_k.
\end{align*}
\end{lemma}
 One easily verifies that
\begin{equation}
\label{eq:G}
\begin{split}
 (\lcc_{E_i} J)E_j &= -\frac{2}{3\sqrt{3}}\ijk (E_k + 2F_k), ~(\lcc_{E_i} J)F_j = -\frac{2}{3\sqrt{3}}\ijk (E_k - F_k), \\
 (\lcc_{F_i} J)E_j &= -\frac{2}{3\sqrt{3}}\ijk (E_k - F_k), ~(\lcc_{F_i} J)F_j =  \phantom{-}\frac{2}{3\sqrt{3}}\ijk (2E_k +F_k).
\end{split}
\end{equation}
The tensor field~$G=\lcc J$ is skew-symmetric, i.e.,  $G(X,Y)+G(Y,X)=(\lcc_{X} J)Y+(\lcc_{Y} J)X=0,~\forall ~X,Y\in TM$, hence~($\nks$, $g$, $J$) is nearly K\"ahler.
Moreover, $G$ satisfies the following properties (cf. \cite{BM2001}, \cite{Gr1}):
\begin{equation}
G(X,JY)+JG(X,Y)=0,~g(G(X,Y),Z)+g(G(X,Z),Y)=0.
\end{equation}

For unitary quaternions~$a$,~$b$ and~$c$, the map~$F\colon \nks\to\nks$ given by
$(p,q)\mapsto (apc^{-1},bqc^{-1})$ is an isometry of $(\nks,g)$ (cf.\ the remark after Lemma~2.2 in \cite{podesta}).
Indeed, $F$ preserves the almost complex structure $J$, since
\begin{align*}
 J dF_{(p,q)}(v,w) &= \frac{1}{\sqrt{3}}\bigl(2(apc^{-1})(cq^{-1}b^{-1})bwc^{-1} - avc^{-1}\bigr.,\\
                         &\quad\quad\quad  \bigl. -2(bqc^{-1})(cp^{-1}a^{-1})avc^{-1}+bwc^{-1}\bigr)\\
                   &= dF_{(p,q)}\bigl(J(v,w)\bigr)
\end{align*}
(see also \cite[Proposition~3.1]{moroianu}) and~$F$ preserves the usual metric~$\metric{\cdot\,}{\cdot}$ as well.

Next, we introduce an almost product structure on the nearly K\"ahler~$\nks$. The~$(1,1)$-tensor field~$P$
is defined by
\begin{equation}
 \label{eq:defP}
  PZ = (pq^{-1}V, qp^{-1}U),
\end{equation}
where~$Z=(U,V)$ is a tangent vector~at~$(p,q)$.
The definition of~$P$ also makes use of the Lie group structure of the nearly K\"ahler~$\nks$.
At the Lie algebra level the map
\[
     T_{(1,1)}\nks\to T_{(1,1)}\nks : (U,V)\mapsto (V,U)
\]
defines an almost product structure.
By left translation this structure can be extended to the manifold~$\nks$, similarly as
was done for the almost complex structure $J$.
We summarize the elementary properties of the almost product in the following lemma.
\begin{lemma}[\cite{bddv}]
\label{lem:P}
The almost product structure~$P$ satisfies the following properties:
\begin{subequations}
\begin{align}
   P^2 &= \mathrm{Id},  \text{ i.e.~$P$ is involutive,}\\
   PJ  &=-JP,           \text{ i.e.~$P$ and~$J$ anti-commute,}  \\
   g(PZ,PZ')&=g(Z,Z'),  \text{ i.e.~$P$ is compatible with~$g$,} \\
   g(PZ,Z') &=g(Z,PZ'), \text{ i.e.~$P$ is symmetric.}
\end{align}
\end{subequations}
\end{lemma}
\begin{proof}
The first three equations can be verified with a direct calculation.
The last equation follows from the first and third equation.
\end{proof}

It is elementary to show that the isometries of ($\nks,~g,~J$) also preserve the almost product structure $P$.
Note that~$PE_i = F_i$ and~$PF_i=E_i$. From these equations and Lemma~\ref{lem:levicivita} it follows that
\begin{equation}
\label{eq:H}
\begin{split}
  (\lcc_{E_i} P)E_j &= \phantom{-}   \frac{1}{3}\ijk(E_k +2F_k),~(\lcc_{E_i} P)F_j =              -\frac{1}{3}\ijk(2E_k +F_k),\\
  (\lcc_{E_i} P)F_j &=              -\frac{1}{3}\ijk(E_k +2F_k),~(\lcc_{F_i} P)F_j = \phantom{-}   \frac{1}{3}\ijk(2E_k +F_k).
\end{split}
 \end{equation}
The tensor field~$\lcc P$ does not vanish identically, so the endomorphism $P$ is not a product structure.
However, the almost product structure~$P$ and its covariant derivative~$\lcc P$ admit the following properties.
\begin{lemma}[\cite{bddv}]
\label{lem:propP}
For tangent vector fields $X$, $Y$ on $(\nks,g,J)$ the following equations hold:
 \begin{gather}
  PG(X,Y) + G(PX,PY) = 0, \label{eq:imp1}\\
  (\lcc_X P)JY = J(\lcc_X P)Y, \label{eq:hj}\\
  G(X,PY) + PG(X,Y) = -2J(\lcc_X P)Y,\label{eq:imp2}\\
  (\lcc_X P)PY + P(\lcc_X P)Y=0,\label{eq:h2}\\
  (\lcc_X P)Y + (\lcc_{PX} P)Y =0,\label{eq:h3}\\
  \overline\nabla P=0.
 \end{gather}
\end{lemma}

 The Riemannian  curvature tensor~$\tilde R$ on $(\nks,g)$ is given by
\begin{equation*}
 \begin{split}
  \tilde R(U,V)W &= \frac{5}{12}\bigl(g(V,W)U - g(U,W)V\bigr) \\
                 &\quad  +\frac{1}{12}\bigl(g(JV,W)JU - g(JU,W)JV - 2g(JU,V)JW\bigr) \\
                 &\quad + \frac{1}{3}\bigl(g(PV,W)PU - g(PU,W)PV  \bigr.\\
                        &\quad  \phantom{\frac{2}{3\sqrt{3}}}\quad\mbox{ } + \bigl. g(JPV,W)JPU - g(JPU,W)JPV\bigr).
 \end{split}
\end{equation*}
One can now show that the nearly K\"ahler~$\nks$ is of constant type~$\frac{1}{3}$ and therefore we have
\begin{align}
g\bigl(G(X,Y),G(Z,W)\bigr) &=\frac{1}{3} \bigl(g(X,Z) g(Y,W)-g(X,W)g(Y,Z)\bigr. \label{eq:inprod}\\
                         &\qquad + \bigl. g(JX,Z)g(JW,Y)-g(JX,W)g(JZ,Y)\bigr), \nonumber\\
G\bigl(X,G(Y,Z)\bigr) &= \frac{1}{3}\bigl( g(X,Z)Y-g(X,Y)Z\bigr. \label{eq:GG}\\
                     & \qquad +\bigl. g(JX,Z)JY -g(JX,Y)JZ\bigr), \nonumber \\
(\lcc G)(X,Y,Z)&= \frac{1}{3} (g(X,Z) JY -g(X,Y) JZ -g(JY,Z)X). \label{eq:nablaG}
\end{align}

For later use, we  also need  the relation between the geometry of the nearly K\"ahler manifold~$(\nks,g)$
and the product manifold~$(\nks,\metric{\cdot}{\cdot})$, which is~$\nks$ endowed with the usual Euclidean product metric.
The equations in this paragraph shall be used every time we want to obtain an explicit parametrization
of a submanifold in the nearly K\"ahler~$\nks$.

The almost product structure $P$ can be expressed in terms of the usual product structure~$QZ= Q(U,V)= (-U,V)$
and vice versa:
\begin{align}
  QZ &= \frac{1}{\sqrt{3}} (2 PJZ - JZ), \label{eq:Q}\\
  PZ &= \frac{1}{2}(Z-\sqrt{3}QJZ). \label{eq:Q2}
\end{align}
Using these equations the Euclidean product metric~$\metric{\cdot}{\cdot}$ can be expressed in
terms of~$g$ and~$P$:
\begin{equation}
\label{eq:metric}
  \metric{Z}{Z'} = \frac{3}{8}\bigl(g(Z,Z')+g(QZ,QZ')\bigr) = g(Z,Z')+ \frac{1}{2}g(Z,PZ'),
\end{equation}
and consequently
\begin{equation}\label{eq:Qmetric}
 \metric{Z}{QZ'} = \frac{\sqrt{3}}{2}g(Z,PJZ').
\end{equation}
We can now show the relation between the Levi-Civita connections~$\lcc$ of~$g$
and~$\ne$ of the  usual Euclidean product metric ~$\metric{\cdot}{\cdot}$ on~$\nks$.
\begin{lemma}[\cite{dlmv}]\label{lem:connection}
The relation between the nearly K\"ahler connection~$\lcc$ and the Euclidean connection~$\ne$ is
 \begin{equation} \label{eq:connection}
  \ne_X Y= \lcc_X Y +\frac{1}{2}\bigl(JG(X,PY) + JG(Y,PX)\bigr).
 \end{equation}
\end{lemma}

\begin{remark}
Using the above lemma and the expression for~$Q$ one can show that~$(\nabla^E_X Q)Y = 0$ implies equation~\eqref{eq:hj}
and vice versa.
In this sense~$P$ really is the ``nearly K\"ahler analogue'' of the Euclidean product structure~$Q$.
\end{remark}

In \cite{schafer},
Sch\"afer and Smoczyk gave a broader study of Lagrangian
submanifolds in a nearly K\"ahler manifold, they also showed that the classical
result of Ejiri \cite{ejiri},  that a Lagrangian submanifold of the nearly K\"ahler~
$\mathbb{S}^6$ is always minimal and orientable, holds actually for
arbitrary 6-dimensional strict nearly K\"ahler~ manifolds (see also \cite{GIP}).
From now on we will assume that $M$ is a Lagrangian submanifold in  the nearly K\"ahler~$\nks$. Hence $M$ is $3$-dimensional and the almost complex structure $J$ maps tangent vectors to normal vectors.
Like Lagrangian submanifolds of the nearly K\"ahler~
$\mathbb{S}^6$, from \cite{GIP} or \cite{schafer} it follows:
\begin{lemma}[cf. \cite{GIP}, \cite{schafer}]\label{lem:lagr} \label{lem:gnorm} Let $M$ be a Lagrangian submanifold of  the nearly K\"ahler~$\nks$. Then $M$ is minimal and orientable.
Moreover, for $X,Y$ tangent to $M$, $G(X,Y)$ is a normal vector field on $M$.
\end{lemma}

If we denote the immersion by $f$, the formulas of Gauss and
Weingarten are respectively given by
\begin{align}
&\widetilde \nabla_X f_\ast Y = f_\ast(\nabla_X Y) + h(X,Y),\label{eqgauss}\\
&\widetilde \nabla_X \eta = - f_\ast(S_\eta X) + \nabla^\perp_X
\eta,\label{eqweingarten}
\end{align}
for tangent vector fields $X$ and $Y$ and a normal vector field $\eta$.
The second fundamental form $h$ is related to $S_\eta$ by
$g(h(X,Y),\eta)=g(S_\eta X,Y)$.
From \eqref{eqgauss} and \eqref{eqweingarten}, we find that
\begin{align}
&\nabla^\perp_X Jf_\ast(Y) = J f_\ast(\nabla_X Y) + G(f_\ast
X,f_\ast Y),\label{normalconnection}\\
&f_\ast(S_{JY} X) = -J h(X,Y).\label{tangentpart}
\end{align}
The above formulas immediately imply that
\begin{equation}
g(h(X,Y),Jf_*Z) =g(h(X,Z),Jf_*Y),\label{symmetryh}
\end{equation}
i.e. $g(h(X,Y),Jf_*Z)$ is totally symmetric.
Of course as usual whenever there is no confusion, we will
drop the immersion $f$ from the notations.

\section{Lagrangian submanifolds of the nearly K\"ahler~$\nks$}
\label{sec:lsnks}

Note that in the previous section most of the results remain valid for Lagrangian submanifolds of arbitrary $6$-dimensional strict nearly K\"ahler manifolds. Here however we will  restrict ourselves to the case that the ambient space is  the nearly K\"ahler~$\nks$. We will show how the properties of the almost product structure $P$, related to the product structure on  the nearly K\"ahler~$\nks$ incorporates most of the geometry of the Lagrangian submanifold.
The key idea is to ``decompose'' the almost product structure~$P$ into a tangent part~$A$ and a normal part~$B$.

Let~$M$ be a Lagrangian submanifold of the nearly K\"ahler~$\nks$. Since~$M$ is Lagrangian,
the pull-back of~$T(\nks)$ to~$M$ splits into~$TM\oplus JTM$.
Therefore there are two endomorphisms~$A,B\colon TM\to TM$ such that the restriction~$P|_{TM}$ of~$P$ to
 the submanifold $M$  equals $A+ JB$, that is
$PX = AX + JBX$ for all~$X\in TM$.
Note that the above formula, together with the fact that $P$ and $J$ anticommute, also determine $P$ on the normal space
by~$PJX = -JPX = BX- JAX$.
The following lemma gives the basic properties of~$A$ and~$B$.

\begin{lemma}
The endomorphisms~$A$ and~$B$ are symmetric commuting endomorphisms
that satisfy~$A^2+B^2 = \Id$.
\end{lemma}
\begin{proof}
The lemma follows easily from the basic properties of~$P$ and $J$ ($P$ is symmetric, $J$ is compatible with $g$).
For~$X,Y\in TM$ we have~$\metri{AX}{Y}=\metri{PX}{Y}= \metri{X}{PY}=\metri{X}{AY}$.
Similarly one finds~$\metri{BX}{Y}=\metri{PJX}{Y}=\metri{PY}{JX}=\metri{JBY}{JX}=\metri{BY}{X}$.
Since~$P$ is involutive, we also have
\[
  X = P^2 X = (A^2+B^2)X + J(BA-AB)X.
\]
Comparing the tangent and normals parts gives~$A^2+B^2 = \Id$ and~$[A,B]=0$.
\end{proof}

As $A$ and $B$ are symmetric operators whose Lie bracket vanishes,
we know that they can be diagonalized  simultaneously at a point of~$M$. Therefore, at each point $p$ there is an orthonormal basis~$e_1$, $e_2$, $e_3\in T_p M$ such that
\begin{align*}
 Pe_i &= \cos 2\theta_i e_i + \sin 2\theta_i Je_i,~\forall~i=1,2,3.
\end{align*}
The factor~$2$ in the arguments of the sines and cosines is there for convenience as it will simplify many of the following expressions.

Now we  extend the orthonormal basis~$e_1$, $e_2$, $e_3$ at a point $p$ to a frame on a
neighborhood of $p$ in the Lagrangian submanifold. By Lemma~1.1-1.2 in~\cite{szabo} the orthonormal basis at a point
can be extended to a differentiable frame~$E_1$,~$E_2$ $E_3$ on an open dense neighborhood
where the multiplicities of the eigenvalues of~$A$ and~$B$ are constant.
Taking also into account the properties of $G$ we know that there exists a local
orthonormal frame $\{E_1,E_2,E_3\}$ on an open dense subset of $M$
such that
\begin{equation}\label{eqn:3.12}
 A E_i= \cos(2 \theta_i ) E_i,~B E_i = \sin(2 \theta_i) E_i,~J G(E_i,E_j)=\tfrac{1}{\sqrt{3}} \varepsilon_{ijk} E_k.
\end{equation}

\begin{lemma}
\label{lem:sumzero}
The sum of the angles~$\theta_1 + \theta_2  + \theta_3$ is zero modulo~$\pi$.
\end{lemma}
\begin{proof}
 Using equation~\eqref{eq:imp1} and (2.4b), we get~$$PE_1 = \sqrt{3}PJG(E_2,E_3) =\sqrt{3} JG(PE_2,PE_3)$$
 and thus ~$\cos 2\theta_1 E_1 + \sin 2 \theta_1 JE_1(=PE_1)$ is equal to
 \[
 \sqrt{3}\Bigl( \cos(2(\theta_2+\theta_3)) JG(E_2,E_3)  +\sin(2(\theta_2+\theta_3)) G(E_2,E_3)\Bigr).
 \]
 Comparing tangent and normal parts gives
 \begin{align*}
  \cos 2\theta_1 &=\cos(2(\theta_2+\theta_3)), ~ \sin 2 \theta_1 =  -\sin(2(\theta_2+\theta_3).
 \end{align*}
 Therefore
 \[
 \cos(2(\theta_1+\theta_2+\theta_3))=\cos 2\theta_1 \cos(2(\theta_2+\theta_3))- \sin 2\theta_1 \sin(2(\theta_2+\theta_3))=1,
 \]
 so~$\theta_1+\theta_2+\theta_3= 0 \mod \pi$.
\end{proof}

Using the decomposition of $P$ and the expression of the curvature tensor of  the nearly K\"ahler~$\nks$ we can now write down the expressions for the equations of Gauss and Codazzi.
We have the equation of Gauss as follows.
\begin{equation}
 \begin{split}
   R(X,Y)Z &= \frac{5}{12}\bigl(g(Y,Z)X - g(X,Z)Y\bigr) \\
                 &\quad + \frac{1}{3}\bigl(g(AY,Z)AX - g(AX,Z)AY  +  g(BY,Z)BX - g(BX,Z)BY\bigr)\\
                 &\quad +S_{h(Y,Z)} X-S_{h(X,Z)} Y.
 \end{split}\label{gauss}
\end{equation}
Note that in view of the symmetry of the second fundamental form the above Gauss equation  can be rewritten as
\begin{equation}
 \begin{split}
   R(X,Y)Z &= \frac{5}{12}\bigl(g(Y,Z)X - g(X,Z)Y\bigr) \\
                 &\quad + \frac{1}{3}\bigl(g(AY,Z)AX - g(AX,Z)AY  +  g(BY,Z)BX - g(BX,Z)BY\bigr)\\
                 &\quad +[S_{JX},S_{JY}] Z.
 \end{split}\label{gauss2}
\end{equation}
By taking the normal part of the curvature tensor, we  have that  the Codazzi equation is given by
\begin{equation}
\begin{split}
(\nabla h)&(X,Y,Z))-(\nabla h) (Y,X,Z)=\\
&\frac{1}{3}\bigl(g(AY,Z)JBX - g(AX,Z)JBY  -  g(BY,Z)JAX + g(BX,Z)JAY\bigr).
\end{split}
\label{codazzi}
\end{equation}

Analogously, like Lagrangian immersions of the nearly K\"ahler $\mathbb{S}^6$, we find that the Ricci equation is equivalent with the Gauss equation. Indeed from \eqref{normalconnection},  \eqref{eq:nablaG}
and the fact that $G(X,Y)$ is a normal vector field we get that
\begin{equation}
R^\perp(X,Y)JZ=J R(X,Y)Z+\tfrac{1}{3} (g(X,Z)JY -g(Y,Z)JX).\label{normalcurvature}
\end{equation}
Therefore by applying  the Gauss equation \eqref{gauss}, we recover that
\begin{equation}
\begin{split}
R^\perp(X,Y)JZ=& \frac{1}{12}\bigl(g(Y,Z)JX - g(X,Z)JY\bigr) \\
                 &\quad + \frac{1}{3}\bigl(g(AY,Z)JAX - g(AX,Z)JAY  +  g(BY,Z)JBX - g(BX,Z)JBY\bigr)\\
                 &\quad +J[S_{JX},S_{JY}] Z.
                 \end{split}
\end{equation}
Hence by taking the inner product with $JW$ we get the Ricci equation
\begin{align*}
g(R^\perp(X,Y)JZ,JW)&= g(\tilde R(X,Y)JZ,JW)+ g([S_{JX},S_{JY}] Z,W)\\
&= g(\tilde R(X,Y)JZ,JW)+ g(S_{JZ},S_{JW}] X,Y).
\end{align*}

We now calculate the covariant derivatives of~$A$ and~$B$.
\begin{lemma}
\label{lem:covAB}
The covariant derivatives of the endomorphisms~$A$ and~$B$ are
\begin{align*}
 (\nabla_X A)Y &= B S_{JX} Y - Jh(X,BY) +\frac{1}{2}\bigl(JG(X,AY)-AJG(X,Y)\bigr),\\
 (\nabla_X B)Y &=  Jh(X,AY) -A S_{JX} Y +\frac{1}{2}\bigl(JG(X,BY)-BJG(X,Y)\bigr).
\end{align*}
\end{lemma}
\begin{proof}
We express equation~\eqref{eq:imp2} in terms of~$A$ and~$B$.
By  the Gauss and Weingarten formula and Lemma~\ref{lem:lagr} we get on one hand
\begin{align*}
 (\lcc_X P)Y &= \lcc_X AY + \lcc_X JBY - P\nabla_X Y - Ph(X,Y)\\
             &=\nabla_X AY + h(X,AY)+ J\lcc_X BY + G(X,BY)\\
             &\quad - A\nabla_X Y - JB\nabla_X Y - PJS_{JX}Y\\
             &= (\nabla_X A)Y + J(\nabla_X B)Y + Jh(X,BY) - BS_{JX}Y \\
             &\quad + h(X,AY) + JAS_{JX}Y + G(X,BY).
\end{align*}
On the other hand we have
\begin{align*}
\frac{1}{2}(JG(X,P&Y)+JPG(X,Y)) \\
        &= \frac{1}{2}( JG(X,AY) +G(X,BY) - AJG(X,Y)-JBJG(X,Y)).
\end{align*}
Using Lemma~\ref{lem:gnorm} we can compare the tangent and normal parts in equation~\eqref{eq:imp2}.
This gives us the covariant derivatives of~$A$ and~$B$.
\end{proof}

It would be interesting to ask whether it is possible to prove an existence and uniqueness theorem like for submanifolds of real space forms or lagrangian submanifolds of complex space forms. Although such a theorem would simplify some of the later proofs, it is outside the scope of the present paper.

For the Levi-Civita connection~$\nabla$ on~$M$ we introduce the functions~$\omega_{ij}^k$
satisfying~$\nabla_{E_i} E_j = \omega_{ij}^k  E_k$ and~$\omega_{ij}^k =- \omega_{ik}^j$, where we have used Einstein summation.
We write~$h_{ij}^k=\metri{h(E_i,E_j)}{JE_k}$.
The tensor~$h_{ij}^k$ is a totally symmetric tensor on the Lagrangian submanifold.
The covariant derivative on the nearly K\"ahler~$\nks$ will be denoted by~$\lcc$ as usual.
In the following, we will use equation~\eqref{eq:imp2} to obtain extra information on the~angles~$\theta_i$ and second
fundamental form~$h_{ij}^k$.
\begin{lemma}
 \label{lem:sff}
 The derivatives of the angles~$\theta_i$ give the components of the second fundamental form
 \[
   E_i(\theta_j) = -h_{jj}^i
 \]
 except~$h_{12}^3$. The second fundamental form and covariant derivative are related by
 \[
   h_{ij}^k \cos(\theta_j-\theta_k) = \Bigl(\frac{\sqrt{3}}{6}\e_{ij}^k-\omega_{ij}^k \Bigr)\sin(\theta_j-\theta_k),~\forall~j\neq k, ~\text{where}~ \e_{ij}^k:=\e_{ijk}.
 \]
\end{lemma}
\begin{proof}
We will not do all the calculations explicitly, instead we give
one calculation as an example. Choose~$X=Y=E_1$ in~\eqref{eq:imp2}.
Then the equation~$2(\lcc_{E_1}P)E_1 = JG(E_1,PE_1)$ gives
\begin{align*}
 -2 \bigl( h_{11}^1 + E_1(\theta_1)\bigr)  \sin(2 \theta_1) &= 0,   \\
  2 \bigl( h_{11}^1 + E_1(\theta_1)\bigr) \cos(2 \theta_1) &= 0, \\
 -2 \bigl( h_{11}^2  \cos(\theta_1-\theta_2) + \omega_{11}^2 \sin(\theta_1-\theta_2) \bigr) \sin(\theta_1+\theta_2) &=0,   \\
  2 \bigl( h_{11}^2  \cos(\theta_1-\theta_2) + \omega_{11}^2 \sin(\theta_1-\theta_2) \bigr) \cos(\theta_1+\theta_2) &=0,  \\
 -2 \bigl( h_{11}^3 \cos(\theta_1-\theta_3) + \omega_{11}^3\sin(\theta_1-\theta_3) \bigr) \sin(\theta_1+\theta_3) &= 0,   \\
  2 \bigl( h_{11}^3 \cos(\theta_1-\theta_3) + \omega_{11}^3\sin(\theta_1-\theta_3) \bigr) \cos(\theta_1+\theta_3) &= 0.
\end{align*}
Since the sines and cosines cannot be zero at the same time, we find that~$E_1(\theta_1)=-h_{11}^1$
and two expressions relating~$\omega_{ij}^k$ and~$h_{ij}^k$.
Doing the same calculations for~$X=E_i$ and~$Y=E_j$ with $i,j=1,2,3$, we get the lemma.
\end{proof}

Note that from Lemma \ref{lem:sff} we have that
$$E_i(\theta_j)=-h_{jj}^i.$$
Therefore we also have the compatibility conditions that
\begin{equation}\label{compatibility}
\begin{aligned}
-E_k(h_{jj}^i) + E_i(h_{jj}^k)&= [E_k,E_i](\theta_j)\\
&=\sum_{\ell=1}^3 (\omega_{ki}^\ell-\omega_{ik}^\ell) E_\ell(\theta_j)\\
&= \sum_{\ell=1}^3 (-\omega_{ki}^\ell+\omega_{ik}^\ell) h_{jj}^\ell.
\end{aligned}
\end{equation}
So we have six additional  independent equations. One can show, using Lemma~\ref{lem:sff}, that the above  equations are equivalent with six of the Codazzi equations. One does not obtain all the equations of Gauss and Codazzi this way, but the compatibility conditions for the~$\theta_i$ are easier to calculate.

\begin{remark}
We note that from Lemma \ref{lem:sff} and Lemma \ref{lem:sumzero}, we obtain that
$h_{11}^i+h_{22}^i+h_{33}^i=-E_i(\theta_1+\theta_2+\theta_3)=0,~\forall ~i=1,2,3$.
 Hence, we obtain a new proof of the fact that $M$ is minimal (see Lemma \ref{lem:gnorm}).
\end{remark}
Another consequence of Lemma~\ref{lem:sff} is

\begin{corollary}
\label{cor:totgeod}
Let~$M$ be a Lagrangian submanifold of the nearly K\"ahler~$\nks$. If $M$ is totally geodesic, then the angles~$\theta_1$,~$\theta_2$ and~$\theta_3$ are constant.
Conversely, if the angles are constant and~$h_{12}^3=0$, then $M$ is totally geodesic.
\end{corollary}

\begin{remark}
We must assume that~$h_{12}^3 =0$ in the converse statement of Corollary \ref{cor:totgeod}.
Examples~\ref{ex:l6}-\ref{ex:l8} in the next section show that this assumption is necessary.
\end{remark}

In the following, we give two other corollaries of Lemma~\ref{lem:sff}.
Lemma \ref{lem:equalangles} gives a sufficient condition for a Lagrangian submanifold in the nearly K\"ahler~$\nks$ to be
totally geodesic. Lemma \ref{lem:angles} gives a necessary condition and shows us that the converse
statement of Lemma \ref{lem:equalangles} also holds.
\begin{lemma}
\label{lem:equalangles}
If two of the angles are equal modulo~$\pi$, then the Lagrangian submanifold is totally geodesic.
\end{lemma}
\begin{proof}
 Without loss of generality we may assume that~$\theta_1=\theta_2 \mod \pi$. It follows from Lemma~\ref{lem:sff}
 that~$h_{11}^i= h_{22}^i$ and~$h_{12}^i=0$ for~$i=1,2,3$. Combining the equations, together with the symmetry of $h$, gives~$h_{11}^1= h_{22}^2=h_{11}^2=h_{22}^1=h_{12}^3=0$
 and by minimality also~$h_{33}^1$ and~$h_{33}^2$ vanish. The three  remaining components are related by~$h_{22}^3=h_{11}^3$
 and~$h_{33}^3=-2h_{11}^3$.
 The compatibility condition for $\theta_1$ with respect to $E_1$ and $E_2$ (take $k=j=1,~i=2$ in \eqref{compatibility}) gives~
 \begin{equation}\label{w123-1}
 (\omega_{12}^3-\omega_{21}^3)h_{11}^3=0.
  \end{equation}
 Now we use the Codazzi equation \eqref{codazzi} applied to $X=E_1$, $Y=E_2$, $Z=E_2$. As $\theta_1$ and $\theta_2$ are equal modulo $\pi$, the  term on the right hand side of \eqref{codazzi}vanishes and so we obtain by taking the component in the direction of $JE_2$ that
 \begin{equation}\label{w123-2}
 (\frac{1}{\sqrt{3}}+\omega_{21}^3-3\omega_{12}^3)h_{11}^3=0.
 \end{equation}
We claim that $h_{11}^3=0 $. If $h_{11}^3\neq 0$, then from \eqref{w123-1} and \eqref{w123-2}
we get $\omega_{21}^3=\omega_{12}^3=\frac{\sqrt{3}}{6}$. It follows from the second equation of Lemma \ref{lem:sff} that
$0=(-\frac{\sqrt{3}}{6}-\omega_{21}^3)\sin(\theta_1-\theta_3)$, taking into account that $\omega_{21}^3=\frac{\sqrt{3}}{6}$ we obtain that
$\sin(\theta_1-\theta_3)=0$, hence $\theta_1=\theta_3+a\pi$, where $a$ is a constant integer. Then using the first
equation of Lemma \ref{lem:sff}, we derive that $h_{33}^2=-E_3(\theta_3)=-E_3(\theta_1)=h_{11}^3$, but we have that $h_{33}^3=-2h_{11}^3$,
 so we get that $h_{11}^3=h_{33}^3=0$, which is a contradiction.
 Thus~$h_{11}^3=0$ and the
 submanifold is totally geodesic.
\end{proof}

\begin{lemma}
\label{lem:angles}
 Consider a totally geodesic Lagrangian
submanifold in the  nearly K\"ahler  $\nks$. After a possible permutation of the
angles and the choice of the angles $2\theta_i$ at an initial point
belonging to the interval $[0,2 \pi)$, we must have one of the
following possibilities:
\begin{equation*}
\begin{array}{lll}
&(1)\ \ (2\theta_1,2\theta_2,2\theta_3) =(\tfrac{4\pi}{3},\tfrac{4\pi}{3},\tfrac{4\pi}{3}) ,\ \ \
&(2)\ \ (2\theta_1,2\theta_2,2\theta_3) =(\tfrac{2\pi}{3},\tfrac{2\pi}{3},\tfrac{2\pi}{3}) ,\\
&(3)\ \ (2\theta_1,2\theta_2,2\theta_3) = (0,0,0),\
&(4)\ \ (2\theta_1,2\theta_2,2\theta_3) = (0,\pi,\pi),\\
&(5)\ \ (2\theta_1,2\theta_2,2\theta_3) =
(\tfrac{\pi}{3},\tfrac{\pi}{3},\tfrac{4\pi}{3}),\
&(6)\ \ (2\theta_1,2\theta_2,2\theta_3) =
(\tfrac{2\pi}{3},\tfrac{5\pi}{3},\tfrac{5\pi}{3}).
\end{array}
\end{equation*}
\end{lemma}
\begin{proof}
The Codazzi equation \eqref{codazzi} gives~
$$
g(AY,Z)BX - g(AX,Z)BY = g(BY,Z)AX -g(BX,Z)AY.
$$
Taking~$X=E_i$ and~$Y=Z=E_j$, this yields~
$\sin\big(2(\theta_i-\theta_j)\big)=0$ for~$i\neq j$. So the
angles~$2\theta_i$ are equal up to an integer multiple of~$\pi$.
Together with Lemma~\ref{lem:sumzero} we deduce that the angles need to be constant, and therefore after a choice at an initial point one obtains the possibilities in the statement.
\end{proof}

\section{Examples of Lagrangian submanifolds in the nearly K\"ahler  ~$\nks$}
\label{sec:exlag}

In this  section we present eight examples (or families of examples) of Lagrangian submanifolds in the nearly K\"ahler  ~$\nks$.
Example 4.8 (a flat  Lagrangian torus) is a new example. Examples 4.1- 4.3 are the factors and the diagonal which were given by Sch\"afer and Smozcyk in \cite{schafer}. Examples 4.4-4.7 were constructed by Moroianu and Semmelmann in~\cite{ms}, where they studied
 generalized Killing spinors on the standard sphere $\mathbb{S}^3$, which turn out to be related to
  Lagrangian embeddings in the nearly K\"ahler  $\nks$. The first seven examples are immersions of round $3$-spheres or Berger spheres. On these $3$-sphere
$\mathbb{S}^3$ as the set of all the unit quaternions in
$\mathbb{H}$, we consider the left invariant tangent vector fields
$X_1,X_2,X_3$ on $\mathbb{S}^3$, which are given by
\begin{equation}\label{eqn:2.66}
 X_1(u)=u\,\mathbf{i},~
 X_2(u)=u\,\mathbf{j},~
 X_3(u)=-u\,\mathbf{k},
\end{equation}
where $u=x_1+x_2\mathbf{i}+x_3\mathbf{j}+x_4\mathbf{k}\in
\mathbb{S}^3$ is viewed as a unit quaternion, and
$\mathbf{i},\,\mathbf{j},\,\mathbf{k}$ are the imaginary units of
$\mathbb{H}$. Obviously, $X_1,X_2,X_3$ form
a basis of the tangent bundle $T\mathbb{S}^3$. We refer  to \cite{ZHDVW} for more details of
Examples 4.1-4.6.

\begin{example}
\label{ex:l1}
 Consider the immersion: $  f\colon \mathbb{S}^3 \to\nks \colon u\mapsto (u,1).$
$f$ is a totally geodesic Lagrangian immersion, $f(\mathbb{S}^3)$ is isometric to a round sphere.
The angles correspond to case (1) of Lemma \ref{lem:angles}.
\end{example}
\begin{example}
\label{ex:l2}
 Consider the immersion:$f\colon \mathbb{S}^3 \to\nks \colon u\mapsto (1,u).$
$f$ is a totally geodesic Lagrangian immersion,  $f(\mathbb{S}^3)$ is isometric to a round sphere. The angles correspond to  case (2) of Lemma \ref{lem:angles}.
\end{example}
\begin{example}
\label{ex:l3}
 Consider the immersion: $ f\colon \mathbb{S}^3 \to\nks \colon u\mapsto (u,u).$
 $f$ is a totally geodesic Lagrangian immersion, $f(\mathbb{S}^3)$ is isometric to a round sphere. The angles correspond to case (3) of Lemma \ref{lem:angles}.
\end{example}
\begin{example}
\label{ex:l4} Consider the immersion $f\colon \mathbb{S}^3\to\nks:$ $u\mapsto (u,ub)$ with~$b\in\im \H$, $\|b\|=1$. First we note that after an
isometry~$(p,q)\mapsto (pa^{-1}, q a^{-1})$ and a
reparametrization~$u\mapsto u a$ of the 3-sphere with~$a\in \im\H$,
the immersion becomes~$u\mapsto(u,uaba^{-1})$.
 We now choose~$a$ such that $aba^{-1}=\mathbf{i}$. This is always
 possible, because conjugation with a unit quaternion gives a
 rotation of~$\im \H$ and the group of rotations acts transitively
 on~$\im \H$. Therefore we may always consider the immersion:
         $f\colon \mathbb{S}^3 \to\nks \colon u\mapsto (u,u\mathbf{i}).$
$f$ is a totally geodesic Lagrangian immersion, $f(\mathbb{S}^3)$ is isometric to a Berger sphere.
The angles correspond to case (4) of Lemma
\ref{lem:angles}.

Note that by changing the parametrization of $\mathbb{S}^3$, we can
also reduce the potential immersion $f(u) = (u\mathbf{i},u)$ to the
preceding example.
\end{example}
\begin{example}
\label{ex:l5} Consider the immersion~$f\colon \mathbb{S}^3\to \nks: u\mapsto
(u,u^{-1}bu)$ with~$b\in\im \H$, $\|b\|=1$. After an isometry of the
 nearly K\"ahler ~$\nks$ and a reparametrization of $u$ as in the previous example,
we can always consider the immersion:
 \[
    f\colon \mathbb{S}^3 \to\nks \colon u\mapsto  (u^{-1},u\mathbf{i}u^{-1}).
 \]
$f$ is a totally geodesic  Lagrangian immersion, $f(\mathbb{S}^3)$ is isometric to a Berger sphere.
The angles correspond to  case (5) of
Lemma \ref{lem:angles}.
\end{example}
\begin{example}
\label{ex:l7} Consider the immersion: $ f\colon \mathbb{S}^3 \to\nks \colon u\mapsto  (u\mathbf{i}u^{-1},u^{-1}).$
$f$ is a totally geodesic  Lagrangian immersion, $f(\mathbb{S}^3)$ is isometric to a Berger sphere. The angles correspond to case (6)
of Lemma \ref{lem:angles}.
\end{example}
\begin{example}
\label{ex:l6}
 Consider the immersion~$f\colon \mathbb{S}^3\to \nks: u\mapsto (uau^{-1},ubu^{-1})$ with unit quaternions~$a,b\in\im \H$ and $\metric{a}{b}=0$.
 After an isometry of the nearly K\"ahler~$\nks$ and a reparametrization we can always consider the immersion
 \[
    f\colon \mathbb{S}^3\to \nks: u\mapsto (u\mathbf{i}u^{-1},u\mathbf{j}u^{-1}).
 \]
 For the tangent map we have~$df(X_1)=(0,2u\mathbf{k}u^{-1})$, $df(X_2)=(-2u\mathbf{k}u^{-1},0)$, $df(X_3)=2(-u\mathbf{j}u^{-1},u\mathbf{i}u^{-1})$.
The inner products are given by ~$g\bigl(df(X_i),df(X_j)\bigr)=\tfrac{16}{3}\delta_{ij}$, so~$f$ is an immersion of a round sphere.
We have that~$Jdf(X_1)=\tfrac{2}{\sqrt{3}}(2,u\mathbf{k}u^{-1})$, $Jdf(X_2)=\tfrac{2}{\sqrt{3}}(u\mathbf{k}u^{-1},-2)$ and
 $Jdf(X_3)=-\tfrac{2}{\sqrt{3}}(u\mathbf{j}u^{-1},u\mathbf{i}u^{-1})$. One can now easily verify that~$f$ is a Lagrangian immersion.
 We also have
 \begin{align*}
 &Pdf(X_1)=(2,0)=-\tfrac{1}{2}(df(X_1)-\sqrt{3}Jdf(X_1)),\\
  &Pdf(X_2)=(0,2)=-\tfrac{1}{2}(df(X_2)+\sqrt{3}Jdf(X_2)),\\ &Pdf(X_3)=-2(u\mathbf{j}u^{-1},-u\mathbf{i} u^{-1})=df(X_3).
  \end{align*}
   The angles~$2\theta_i$ are thus equal to~$0$, $\tfrac{2\pi}{3}$ and~$\tfrac{4\pi}{3}$. Therefore by
 Lemma~\ref{lem:angles} this immersion is not totally geodesic. Since the angles are constant, $h_{12}^3$ is the only non-zero
 component of the second fundamental form. This example shows that we cannot omit the condition~$h_{12}^3=0$
 in Corollary~\ref{cor:totgeod}.
\end{example}
\begin{example}
\label{ex:l8}
 Consider the immersion~$f: \mathbb R^3\to \nks: (u,v,w)\mapsto (p(u,w),q(u,v))$ where $p$ and $q$ are constant mean curvature tori in $\mathbb{S}^3$ given by
 \tiny
\begin{align*}
p(u,w)&=\left (\cos \left(\frac{\sqrt{3} u}{2}\right) \cos \left(\frac{\sqrt{3} w}{2}\right),\cos \left(\frac{\sqrt{3} u}{2}\right) \sin \left(\frac{\sqrt{3}
   w}{2}\right),\sin \left(\frac{\sqrt{3} u}{2}\right) \cos \left(\frac{\sqrt{3} w}{2}\right),\sin \left(\frac{\sqrt{3} u}{2}\right) \sin \left(\frac{\sqrt{3}
   w}{2}\right)\right),\\
  q(u,v)&=\frac{1}{\sqrt{2}}\left (\cos \left(\frac{\sqrt{3} v}{2}\right) \left(\sin \left(\frac{\sqrt{3} u}{2}\right)+\cos \left(\frac{\sqrt{3} u}{2}\right)\right),\sin
   \left(\frac{\sqrt{3} v}{2}\right) \left(\sin \left(\frac{\sqrt{3} u}{2}\right)+\cos \left(\frac{\sqrt{3} u}{2}\right)\right)\right . ,\\
   &\qquad \left .  \cos \left(\frac{\sqrt{3}
   v}{2}\right) \left(\sin \left(\frac{\sqrt{3} u}{2}\right)-\cos \left(\frac{\sqrt{3} u}{2}\right)\right),\sin \left(\frac{\sqrt{3} v}{2}\right) \left(\sin
   \left(\frac{\sqrt{3} u}{2}\right)-\cos \left(\frac{\sqrt{3} u}{2}\right)\right) \right).
\end{align*}
\normalsize
It follows that
\tiny
\begin{align*}
&f_u=\left (\left (-\frac{\sqrt{3}}{2}  \sin \left(\tilde{u}\right) \cos \left(\tilde{w}\right),-\frac{\sqrt{3}}{2}  \sin \left(\tilde{u}\right) \sin
   \left(\tilde{w}\right),\frac{\sqrt{3}}{2}  \cos \left(\tilde{u}\right) \cos \left(\tilde{w}\right),\frac{\sqrt{3}}{2}  \cos \left(\tilde{u}\right) \sin
   \left(\tilde{w}\right)\right ),\right.
   \\
   &\qquad \left . \left (\frac{1}{2} \sqrt{\frac{3}{2}} \cos \left(\tilde{v}\right) \left(\cos \left(\tilde{u}\right)-\sin \left(\tilde{u}\right)\right),\frac{1}{2}
   \sqrt{\frac{3}{2}} \sin \left(\tilde{v}\right) \left(\cos \left(\tilde{u}\right)-\sin \left(\tilde{u}\right)\right),\right . \right.\\
   &\qquad \left . \left. \frac{1}{2} \sqrt{\frac{3}{2}} \cos
   \left(\tilde{v}\right) \left(\sin \left(\tilde{u}\right)+\cos \left(\tilde{u}\right)\right),\frac{1}{2} \sqrt{\frac{3}{2}} \sin \left(\tilde{v}\right)
   \left(\sin \left(\tilde{u}\right)+\cos \left(\tilde{u}\right)\right)\right ) \right ),\\
&f_v=\left (\left  (0,0,0,0 \right ), \left (-\frac{1}{2} \sqrt{\frac{3}{2}} \sin \left(\tilde{v}\right) \left(\sin \left(\tilde{u}\right)+\cos \left(\tilde{u}\right)\right),\frac{1}{2}
   \sqrt{\frac{3}{2}} \cos \left(\tilde{v}\right) \left(\sin \left(\tilde{u}\right)+\cos \left(\tilde{u}\right)\right),\right. \right .\\
   &\qquad \left . \left .-\frac{1}{2} \sqrt{\frac{3}{2}} \sin
   \left(\tilde{v}\right) \left(\sin \left(\tilde{u}\right)-\cos \left(\tilde{u}\right)\right),\frac{1}{2} \sqrt{\frac{3}{2}} \cos \left(\tilde{v}\right)
   \left(\sin \left(\tilde{u}\right)-\cos \left(\tilde{u}\right)\right)\right ) \right ),\\
&f_w=\left ( \left (-\frac{\sqrt{3}}{2} \cos \left(\tilde{u}\right) \sin \left(\tilde{w}\right),\frac{\sqrt{3}}{2}\cos \left(\tilde{u}\right) \cos
   \left(\tilde{w}\right),-\frac{\sqrt{3}}{2} \sin \left(\tilde{u}\right) \sin \left(\tilde{w}\right),\frac{\sqrt{3}}{2}  \sin \left(\tilde{u}\right) \cos
   \left(\tilde{w}\right)\right ),\left ( 0,0,0,0\right ) \right),
\end{align*}
\normalsize
where in order to simplify expressions we have written $\tilde u = \tfrac{\sqrt{3}}{2} u$,  $\tilde v = \tfrac{\sqrt{3}}{2} v$ and $\tilde w =\tfrac{\sqrt{3}}{2} w$.
A straightforward computations gives that
\begin{align*}
&Jf_u=\left ( \frac{1}{2} \left (-\sin \left(\tilde{u}\right) \cos \left(\tilde{w}\right),-\sin \left(\tilde{u}\right) \sin \left(\tilde{w}\right),\cos
   \left(\tilde{u}\right) \cos \left(\tilde{w}\right),\cos \left(\tilde{u}\right) \sin \left(\tilde{w}\right)\right ) ,\right .\\
   &\qquad \left . \frac{1} {2 \sqrt{2}} \left (\cos \left(\tilde{v}\right) \left(\sin \left(\tilde{u}\right)-\cos \left(\tilde{u}\right)\right),\sin \left(\tilde{v}\right) \left(\sin
   \left(\tilde{u}\right)-\cos \left(\tilde{u}\right)\right),\right. \right.\\
   &\qquad \left. \left. -\cos \left(\tilde{v}\right) \left(\sin \left(\tilde{u}\right)+\cos
   \left(\tilde{u}\right)\right),-\sin \left(\tilde{v}\right) \left(\sin \left(\tilde{u}\right)+\cos \left(\tilde{u}\right)\right)\right ) \right), \\
&Jf_v=\left ( \left (-\sin \left(\tilde{u}\right) \sin \left(\tilde{w}\right),\sin \left(\tilde{u}\right) \cos \left(\tilde{w}\right), \cos \left(\tilde{u}\right) \sin
   \left(\tilde{w}\right),-\cos \left(\tilde{u}\right) \cos \left(\tilde{w}\right)\right),\right .
\\ &\qquad \left .  \frac{1} {2 \sqrt{2}} \left (-\sin \left(\tilde{v}\right) \left(\sin \left(\tilde{u}\right)+\cos \left(\tilde{u}\right)\right),\cos \left(\tilde{v}\right) \left(\sin
   \left(\tilde{u}\right)+\cos \left(\tilde{u}\right)\right),\right . \right .
\\ &\qquad \left . \left .\sin \left(\tilde{v}\right) \left(\cos \left(\tilde{u}\right)-\sin
   \left(\tilde{u}\right)\right),\cos \left(\tilde{v}\right) \left(\sin \left(\tilde{u}\right)-\cos \left(\tilde{u}\right)\right)\right )\right),\\
&Jf_w=\left( \frac{1}{2} \left (\cos \left(\tilde{u}\right) \sin \left(\tilde{w}\right),-\cos \left(\tilde{u}\right) \cos \left(\tilde{w}\right),\sin \left(\tilde{u}\right) \sin
   \left(\tilde{w}\right),-\sin \left(\tilde{u}\right) \cos \left(\tilde{w}\right)\right),\right .
\\ &\qquad  \left .\frac{1}{\sqrt{2}} \left (\sin \left(\tilde{v}\right) \left(\cos \left(\tilde{u}\right)-\sin \left(\tilde{u}\right)\right),\cos \left(\tilde{v}\right) \left(\sin
   \left(\tilde{u}\right)-\cos \left(\tilde{u}\right)\right),\right . \right .
\\ &\qquad \left . \left .\sin \left(\tilde{v}\right) \left(\sin \left(\tilde{u}\right)+\cos
   \left(\tilde{u}\right)\right),-\cos \left(\tilde{v}\right) \left(\sin \left(\tilde{u}\right)+\cos \left(\tilde{u}\right)\right)\right )\right ).
\end{align*}
From this we get that $f$ is a Lagrangian immersion and that  $\{f_u,f_v,f_w\}$ is an orthonormal basis of the tangent space. Hence it is a flat Lagrangian torus. By a lengthy but straightforward computation we also get that
\begin{align*}
Pf_u= f_u,~Pf_v = -\tfrac 12 f_v+ \tfrac{\sqrt{3}}{2}Jf_v,~Pf_w=  -\tfrac 12 f_w-\tfrac{\sqrt{3}}{2} Jf_w.
\end{align*}
The angles~$2\theta_i$ are therefore again  equal to~$0$, $\tfrac{2\pi}{3}$ and~$\tfrac{4\pi}{3}$. Therefore by
 Lemma~\ref{lem:angles} this immersion is also not totally geodesic. Since the angles are constant, $h_{12}^3$ is again the only non-zero
 component of the second fundamental form. This example is another example that shows that we cannot omit the condition~$h_{12}^3=0$
 in Corollary~\ref{cor:totgeod}.
\end{example}

\section{Lagrangian submanifolds of constant sectional curvature}
\label{sec:lscc}

In this section we classify all Lagrangian submanifolds  of constant  sectional curvature in the nearly K\"ahler~$\nks$ .
We will prove that those Lagrangian submanifolds of the nearly K\"ahler~$\nks$ are
congruent with one of the examples  of constant  sectional curvature listed in the previous section.
 As a corollary,
we obtain that the radius of a round Lagrangian sphere in the nearly K\"ahler $\nks$  can only be $\frac{2}{\sqrt{3}}$ or $\frac{4}{\sqrt{3}}$.
This improves Proposition~4.4 of~\cite{ms}.

In order to prove the classification, the first step is to find all the components~$h_{ij}^k$ of the second fundamental form. As
we have already obtained the complete classification of the totally geodesic Lagrangian
submanifolds in the nearly k\"ahler $\nks$ in \cite{ZHDVW} (see Theorem \ref{thmtg}), we can assume now that the immersion is not totally geodesic.
 Then from Lemma \ref{lem:equalangles} we may assume that all the angle functions are different (modulo~$\pi$).
 Therefore, we have that there exists a local
orthonormal frame $\{E_1,E_2,E_3\}$ on an open dense subset of $M$
such that \eqref{eqn:3.12} holds.

We note that it is not possible to follow the approach introduced by Ejiri for studying Lagrangian submanifolds  of constant  sectional curvature in the complex space forms (\cite{ejirilagr}) or in the nearly K\"ahler 6-sphere (\cite{ejiri}).
Indeed the Gauss equations give quadratic equations for the~$h_{ij}^k$ and it turns out that
these are not easy to solve directly without additional information. We therefore use another approach.
The next lemma  gives us linear equations for the components~$h_{ij}^k$.
The key idea is to calculate the expression $x$ given by
\begin{equation}
 \label{eq:cyclic}
  x=  3 \underset{{WXY}}{\mathfrak{S}} \left ((\nabla^2 h)(W,X,Y,Z) - (\nabla^2 h)(W,Y,X,Z)\right ),
\end{equation}
where~$\mathfrak{S}$ stands for the cyclic sum, in two different ways.
On one hand we can calculate this using the covariant derivative of the Codazzi equation \eqref{codazzi},
which tells us that $x$ equals the expression~\eqref{eq:longeq}.
On the other hand we can rewrite $x$ as
\begin{equation}
  x=  3\underset{{WXY}}{\mathfrak{S}} \left ((\nabla^2 h)(W,X,Y,Z) - (\nabla^2 h)(X,W,Y,Z)\right ),
\end{equation}
and then by applying the Ricci identity we obtain that this expression $x$ vanishes.
\medskip

More precisely, we have the following key lemma.

\begin{lemma}\label{keylemma}
Let~$M$ be a Lagrangian submanifold of constant  sectional curvature in the nearly K\"ahler~$\nks$  .
Then for all tangent vector fields~$W,X,Y,Z \in TM$ the expression
\begin{align}
\label{eq:longeq}
\begin{split}
\underset{{WXY}}{\mathfrak{S}}\Bigl\{ &\Bigl\{ \metri{JG(Y,W)}{AZ}+\frac{1}{2}\metri{JG(Y,Z)}{AW} -\frac{1}{2}\metri{JG(W,Z)}{AY} \Bigr. \\
       &\quad +\metri{h(W,Z)}{JBY} -\metri{h(Y,Z)}{JBW} \Bigr\} JBX \\
+\mbox{ }&\Bigl\{ \metri{JG(W,Y)}{BZ}+\frac{1}{2}\metri{JG(W,Z)}{BY} -\frac{1}{2}\metri{JG(Y,Z)}{BW}\Bigr. \\
       &\quad +\metri{h(W,Z)}{JAY} -\metri{h(Y,Z)}{JAW} \Bigr\} JAX \\
+\mbox{ }& \metri{AX}{Z}\Bigl\{ JBJG(W,Y) + \frac{1}{2}G(Y,BW) - \frac{1}{2}G(W,BY) \Bigr.\\
       &\qquad \qquad +\mbox{}\Bigl.h(W,AY) - h(Y,AW) \Bigr\} \\
+\mbox{ }& \metri{BX}{Z}\Bigl\{ -JAJG(W,Y) + \frac{1}{2}G(W,AY) - \frac{1}{2}G(Y,AW) \Bigr.\\
       &\qquad \qquad  +\mbox{}\Bigl.h(W,BY) - h(Y,BW) \Bigr\}\Bigl\}
\end{split}
\end{align}
is zero.
\end{lemma}
\begin{proof}
As we have mentioned before, we  calculate expression~\eqref{eq:cyclic} in two different ways.

First, we  calculate $x$ using the covariant derivative of the Codazzi equation \eqref{codazzi},
which gives us the long expression~\eqref{eq:longeq}.
We denote the normal part~$(\tilde R(X,Y)Z)^\perp$ as~$T_1(X,Y,Z)$. This is the righthandside of the Codazzi equation \eqref{codazzi}. So we have that
\begin{equation}\label{5.4}
 \begin{aligned}
 3&((\nabla h)(X,Y,Z))-(\nabla h) (Y,X,Z))=3 T_1(X,Y,Z) \\
 &= \metri{AY}{Z}JBX - \metri{AX}{Z}JBY +\metri{BX}{Z}JAY - \metri{BY}{Z}JAX.
 \end{aligned}
\end{equation}
Using Lemma \ref{lem:lagr} and  \eqref{normalconnection},  the covariant derivative~$\nabla T_1$, where~$\nabla$ is the covariant derivative on~$M$, can be written as
\begin{equation}\label{5.5}
3 (\nabla T_1)(W,X,Y,Z) = T_2(W,X,Y,Z) - T_2(W,Y,X,Z),
\end{equation}
where
\begin{align}
\label{eq:T2}
 \begin{split}
 T_2(W,X,Y,Z) &=\metri{(\nabla_W A)Y}{Z}JBX +\metri{(\nabla_W B)X}{Z}JAY\\
 &\qquad +\metri{AY}{Z} G(W,BX) +g(BX,Z) G(W,AY)\\
 &\qquad +g(AY,Z) J(\nabla_W B)X  +g(BX,Z) J(\nabla_WA)Y.
\end{split}
\end{align}
By Lemma~\ref{lem:lagr} and Lemma \ref{lem:covAB} the tensor~$T_2$ can be expressed completely in terms of~$A$ and~$B$ in the following way:
\begin{equation} \label{intermediate}
\begin{split}
T_2(W,X,Y,Z) &=g(B S_{JW} Y,Z)JBX +g(h(W,BY),JZ)JBX)\\
&-\tfrac 12 g(G(W,AY),JZ) JBX +\tfrac 12 g(G(W,Y), JAZ) JBX\\
&-g(AY,Z) h(W,AX)-g(AY,Z) J A S_{JW}X\\
&+\tfrac 12 g(AY,Z)G(W,BX)- \tfrac 12 g(AY,Z) JBJ G(W,X)\\
&-g(h(W,AX),JZ) JAY -g(S_{JW} X,AZ) JAY\\
&-\tfrac 12 g(G(W,BX),JZ) JAY + \tfrac 12 g(G(W,X),JBZ) JAY\\
&+g(BX,Z) JB S_{JW}Y +g(BX,Z) h(W,BY)\\
&+\tfrac 12 g(BX,Z) G(W,AY) -\tfrac 12 g(BX,Z) JAJ G(W,Y).
\end{split}
\end{equation}

Now we can compute $x$. From \eqref{eq:cyclic}, \eqref{5.4} and \eqref{5.5}, we have that
\begin{equation*}\begin{split}
x&= T_2(W,X,Y,Z) + T_2(X,Y,W,Z)+ T_2(Y,W,X,Z)\\
&\qquad -T_2(W,Y,X,Z)- T_2(X,W,Y,Z)- T_2(Y,X,W,Z).
\end{split}
\end{equation*}
Therefore when we compute $x$ we can omit in \eqref{intermediate} all terms which are symmetric in two of the variables $W$, $X$, $Y$. So we get $x$ by omitting these terms in \eqref{intermediate} and by taking the cyclic sum of the difference of remainder of \eqref{intermediate} with itself with two variables interchanged.
Hence we can write
\begin{equation*}\begin{split}
x&= T_3(W,X,Y,Z) + T_3(X,Y,W,Z)+ T_3(Y,W,X,Z)\\
&\qquad -T_3(W,Y,X,Z)- T_3(X,W,Y,Z)- T_3(Y,X,W,Z),
\end{split}
\end{equation*}
where
\begin{equation*} \label{intermediate2}
\begin{split}
T_3(W,X,Y,Z) &= g(h(W,BY),JZ)JBX)
-\tfrac 12 g(G(W,AY),JZ) JBX \\
&+\tfrac 12 g(G(W,Y), JAZ) JBX-g(AY,Z) h(W,AX)\\
&+\tfrac 12 g(AY,Z)G(W,BX)- \tfrac 12 g(AY,Z) JBJ G(W,X)\\
&-g(h(W,AX),JZ) JAY -\tfrac 12 g(G(W,BX),JZ) JAY \\
&+ \tfrac 12 g(G(W,X),JBZ) JAY+g(BX,Z) h(W,BY)\\
&+\tfrac 12 g(BX,Z) G(W,AY) -\tfrac 12 g(BX,Z) JAJ G(W,Y).
\end{split}
\end{equation*}
From this we immediately get that $x$ equals the expression~\eqref{eq:longeq}.

Next, we can rewrite $x$ as
\begin{equation}
 \label{eq:cyclicx}
  x=  3\underset{{WXY}}{\mathfrak{S}} \left ((\nabla^2 h)(W,X,Y,Z) - (\nabla^2 h)(X,W,Y,Z)\right ).
\end{equation}
By the Ricci identity, we have that \begin{equation}
\label{eq:cyclic2}
 x= 3 \underset{{WXY}}{\mathfrak{S}} \left( R^\perp(W,X)h(Y,Z) - h\bigl(R(W,X)Y,Z\bigr) - h\bigl(Y,R(W,X)Z\bigr)\right).
\end{equation}
Equations~\eqref{tangentpart} and~\eqref{normalcurvature} give
\begin{align*}
 R^\perp (W,X) &h(Y,Z) \\
    &= R^\perp (W,X) JS_{JY}Z \\
    &= JR(W,X)S_{JY}Z + \frac{1}{3}\bigl(\metri{h(W,Y)}{JZ}JX - \metri{h(X,Y)}{JZ}JW\bigr).
\end{align*}
Since~$M$ has constant curvature the curvature tensor (we denote the constant by $c$), we have $R(X,Y)Z = c(g(Y,Z)X -g(X,Z)Y)$.
An easy calculation shows that $x$ vanishes.
This completes the proof of the lemma.
\end{proof}

We are now in a position to prove the   classification result.
We consider again the endomorphisms~$A$ and~$B$ that satisfy~$P|_{TM}=A+JB$
and take the orthonormal basis~$E_1,E_2,E_3$ such that~$AE_i = \lambda_i E_i$ and~$BE_i = \mu_i E_i$
for~$i=1,2,3$. In the notation of the previous sections~$\lambda_i= \cos 2\theta_i$
and~$\mu_i=\sin 2\theta_i$.
As sometimes the expressions in terms of~$\lambda_i$ and~$\mu_i$
are shorter, so  we will not always express equations in terms of the angles~$\theta_i$.
Taking into account of the properties of $G$, we may also assume that~$J G(E_1,E_2)=\frac{\sqrt{3}}{3} E_3$
by replacing~$E_3$ by~$-E_3$ if necessary. Thus we obtain that  $J G(E_i,E_j)=\tfrac{1}{\sqrt{3}} \varepsilon_{ijk} E_k$.
By taking~$X=E_1$,~$Y=E_2$ and~$Z=W=E_3$ in formula~\eqref{eq:longeq} in Lemma \ref{keylemma}, we obtain six equations, namely
\begin{align}
 &\bigl(\lambda_i(\lambda_j-\lambda_k)+\mu_i(\mu_j-\mu_k)\bigr) h_{kk}^j +
  \bigl(\lambda_k(\lambda_i-\lambda_j)+\mu_k(\mu_i-\mu_j)\bigr) h_{ii}^j =0, \label{eq:lm1}\\
 &\bigl(\lambda_i(\lambda_j-\lambda_k)+\mu_i(\mu_j-\mu_k)\bigr) h_{12}^3 =0, \label{eq:lm2}
\end{align}
for every positive permutation~$(ijk)$ of~$(123)$. Only four of the above equations are linearly independent.
\medskip

We now distinguish two cases: \textbf{Case~1: $h_{12}^3\neq 0$} and \textbf{Case~2: $h_{12}^3=0$.}
\medskip

\textbf{Case~1:} $\mathbf{h_{12}^3\neq 0}$. First we note that
\begin{align*}
\lambda_1(\lambda_2-\lambda_3) +\mu_1(\mu_2-\mu_3)&=\cos(2(\theta_1-\theta_2))-\cos(2(\theta_1-\theta_3))\\
&=-2 \sin (2 \theta_1 -\theta_2-\theta_3) \sin(\theta_3 -\theta_2).
\end{align*}
So from equation~\eqref{eq:lm2} we find
that
$\sin (2 \theta_1 -\theta_2-\theta_3) \sin(\theta_3 -\theta_2)$, $\sin (2 \theta_2 -\theta_3-\theta_1) \sin(\theta_1 -\theta_3)$ and $\sin (2 \theta_3 -\theta_2-\theta_3) \sin(\theta_2 -\theta_1)$ have to vanish. As the immersion is not totally geodesic, from Lemma \ref{lem:equalangles} we have that the angle functions are mutually different which in turn implies that for $i$ different from $j$, we have that $\sin(\theta_j-\theta_i)$ is different from $0$.
Hence
\begin{equation*}
\sin (2 \theta_1 -\theta_2-\theta_3)=\sin (2 \theta_2 -\theta_3-\theta_1)=\sin (2 \theta_3 -\theta_2-\theta_3)=0.
\end{equation*}
So $(2 \theta_1 -\theta_2-\theta_3)$ is a multiple of $\pi$.
By Lemma~\ref{lem:sumzero}, $( \theta_1 +\theta_2+\theta_3)$ is also a multiple of  $\pi$. Hence $\theta_1$ is a multiple of  $\tfrac{\pi}{3}$. A same argument can be applied for the other angles $\theta_2$ and $\theta_3$.
As  the immersion is not totally geodesic,  from Lemma \ref{lem:equalangles} we have that no two angle functions are the same and therefore the angles must be different modulo~$\pi$. So the only possibility for the angles are
~$\mathbf{0}$, $\mathbf{\tfrac{\pi}{3}}$ and~$\mathbf{\tfrac{2\pi}{3}}$.

Since all the angles~$\theta_i$ are constant all the~$h_{jj}^i$ are zero
except~$h_{12}^3$ by Lemma~\ref{lem:sff}.
By~Lemma~\ref{lem:sff} it now follows that all the connection coefficients~$\omega_{ij}^k$ are zero except the cases that $i,j,k$ are all different.
These non-zero coefficients~$\omega_{ij}^k$ can be written in terms of~$h_{12}^3$ in the following way
\begin{equation}\omega_{12}^3= \omega_{23}^1 = \omega_{31}^2= \tfrac{\sqrt{3}}{3} h_{12}^3 + \tfrac{\sqrt{3}}{6}.\end{equation}
Note also that from  the Gauss equation \eqref{gauss} it follows that the constant curvature $c$ is related to the second fundamental form
by
\begin{equation}\label{ce2-1}
\begin{aligned}
c E_2&= R(E_2,E_1)E_1\\
&=(\tfrac{5}{12}-\tfrac{1}{6}) E_2+[S_{JE_2},S_{JE_1}] E_1\\
&= \tfrac 14 E_2-S_{JE_1} (h_{12}^3 E_3)\\
&=\tfrac 14 E_2 -(h_{12}^3)^2 E_2.
\end{aligned}
\end{equation}
This implies that $h_{12}^3$ and therefore also $\omega_{12}^3$, $\omega_{23}^1$ and $\omega_{31}^2$ are all constants. So computing the curvature by the definition we have that
\begin{equation}\label{ce2-2}
\begin{aligned}
c E_2 &=R(E_2,E_1)E_1 \\
&=\nabla_{E_2} \nabla_{E_1} E_1 - \nabla_{E_1} \nabla_{E_2} E_1 -\nabla_{[E_2,E_1]}E_1\\
&=-\omega_{21}^3 \omega_{13}^2 E_2 -(\omega_{21}^3-\omega_{12}^3) \omega_{31}^2 E_2\\
&=(\omega_{12}^3)^2 E_2\\
&=(\tfrac{\sqrt{3}}{3} h_{12}^3 + \tfrac{\sqrt{3}}{6})^2 E_2.
\end{aligned}
\end{equation}
Comparing both expressions \eqref{ce2-1} and \eqref{ce2-2}, we get that~$8 (h_{12}^3)^2+2h_{12}^3=1$, which implies that ~$h_{12}^3=\tfrac{1}{4}$ or~$-\tfrac{1}{2}$.
In the following, we will discuss two subcases of case 1 respectively:
\textbf{Case~1a: $h_{12}^3=\tfrac{1}{4}$} and \textbf{Case~1b: $h_{12}^3=-\tfrac{1}{2}$}.

\textbf{Case~1a:} $\mathbf{h_{12}^3=\tfrac{1}{4}}$. In this case, we have that ~$\omega_{12}^3=\omega_{23}^1=\omega_{31}^2=\frac{\sqrt{3}}{4}$
and the sectional curvature is equal to~$\tfrac{3}{16}$.
\medskip

In the next theorem we will prove that in this  case (\textbf{Case~1a})   the submanifold $M$  is locally congruent with
the immersion in Example~\ref{ex:l6}.
In order to prove this, we first recall that the Berger sphere
can be constructed by looking at $\mathbb{S}^3$ as a hypersurface of the quaternions. As before we take the frame~$X_1(u)=u\mathbf{i}$, $X_2(u)=u\mathbf{j}$, $X_3(u)=-u\mathbf{k}$ of left invariant vector fields. It follows by a straightforward calculation that
$$[X_1,X_2]=-2 X_3,\quad [X_2,X_3]=-2 X_1,\quad [X_3,X_1]=-2 X_2.$$
We now define a new metric $g_b$, depending on two constants $\tau$ and $\kappa$ on $\mathbb{S}^3$ by
\begin{equation*}
g_b(X,Y)=\frac{4}{\kappa} \left (\langle X,Y\rangle    +(\frac{4 \tau^2}{\kappa}-1) \langle  X,X_1\rangle   \langle  Y,X_1\rangle   \right ).
\end{equation*}
This implies that the vector fields $E_1=\tfrac{\kappa}{4 \tau} X_1$, $E_2 = \tfrac{\sqrt{\kappa}}{2} X_2$ and  $E_3 = \tfrac{\sqrt{\kappa}}{2} X_3$ form an orthonormal basis of the tangent space with respect to $g_b$. It follows immediately from the Koszul formula that $\nabla_{E_i}E_i = 0$ and that
\begin{equation}\label{nablaei}
\begin{aligned}
&\nabla_{E_2} E_3 = -\tau E_1, \qquad &&\nabla_{E_2} E_1 = \tau E_3,\\
&\nabla_{E_3} E_2 = \tau E_1, \qquad &&\nabla_{E_3} E_1 = -\tau E_2,\\
&\nabla_{E_1} E_2 =(\tau -\frac{\kappa}{2 \tau}) E_3, \qquad &&\nabla_{E_1} E_3 = (-\tau +\frac{\kappa}{2 \tau})E_2.
\end{aligned}
\end{equation}
Note that the following theorem of \cite{classification}
which can be proved similarly to the local version of the Cartan-Ambrose-Hicks theorem (cf. the proof of Theorem 1.7.18 of \cite{wolf}),
then shows that a manifold admitting such vector fields is locally isometric with a Berger sphere.
\begin{proposition} \label{propberger} Let $M^n$ and $\tilde M^n$ be Riemannian manifolds with Levi-Civita connections $\nabla$ and $\tilde \nabla$. Suppose that there exists constant $c_{ij}^k$, $i,j,k \in \{1,\dots, n\}$ such that for all $p \in M$ and $\tilde p \in \tilde M$ there exist orthonormal frame fields $\{E_1,\dots, E_n\}$ around $p$ and $\{\tilde E_1,\dots, \tilde E_n\}$ around $\tilde p$ such that $\nabla_{E_i} E_j = \sum_{k=1}^n c_{ij}^k E_k$ and $\nabla_{\tilde E_i} \tilde E_j = \sum_{k=1}^n c_{ij}^k \tilde E_k$. Then for every point $p \in M$ and $\tilde p \in \tilde M$ there exists a local isometry which maps a neighborhood of $p$ onto a neighborhood of $\tilde p$ and $E_i$ on $\tilde E_i$.
\end{proposition}
The previous proposition can of course be also applied in case that $\kappa=4 \tau^2$. In that case we simply have a regular sphere  of constant  sectional curvature.

\begin{theorem}\label{thmeg16} Let $M$ be a Lagrangian submanifold of  the nearly K\"ahler~$\nks$. Assume that there exists a local orthonormal frame as in \textbf{Case~1a}.
Then $M$ is locally congruent with the immersion  $f\colon \mathbb{S}^3\to \nks: u\mapsto  (u\mathbf{i}u^{-1},u\mathbf{j}u^{-1})$, which is Example \ref{ex:l6}.
\end{theorem}
\begin{proof}

We have that
$$\omega_{ij}^k =\tfrac{\sqrt{3}}{4} \epsilon_{ij}^k$$
and the only non vanishing component of the second fundamental form is $$g(h(E_1,E_2), JE_3)= \tfrac{1}{4}.$$
This implies immediately that $M$ is congruent with a space of constant sectional curvature $\tfrac{3}{16}$.
Moreover, from the beginning of the
 discussion about \textbf{Case 1}, we know that the angle functions are given by $(2 \theta_1,2 \theta_2,2 \theta_3)=(0,\tfrac{2\pi}{3}, \tfrac{4\pi}{3})$.
 So we can find a local basis such that $\sqrt{3} JG(E_1,E_2)=E_3$ and
\begin{align*}
PE_1=E_1,~
PE_2=-\tfrac 12 E_2 +\tfrac{\sqrt{3}}{2} JE_2,~
PE_3=-\tfrac 12 E_3 -\tfrac{\sqrt{3}}{2} JE_3.
\end{align*}
From this and \eqref{eq:Q} it follows that
\begin{align*}
QE_1=-\sqrt{3} JE_1,~
QE_2= E_2,~
QE_3= -E_3.
\end{align*}
Applying Proposition \ref{propberger} and comparing with \eqref{nablaei} (take $\kappa=\frac{3}{4},~\tau=\frac{\sqrt{3}}{4}$), we  have that  we can identify $M$ with $\mathbb{S}^3$, with a proportional metric and that we may assume that
\begin{align*}
E_1 = -\tfrac{\sqrt{3}}{4} X_3,~
E_2 = -\tfrac{\sqrt{3}}{4} X_1,~
E_3 = -\tfrac{\sqrt{3}}{4} X_2.
\end{align*}

We now write the immersion $f= (p,q)$ and $df(E_i)= D_{E_i} f = (p \alpha _i, q \beta_i)$ where $\alpha_i,\beta_i$ are imaginary quaternions. In view of the above properties of $Q$, it immediately follows that $\beta_1= \alpha_1$, $\alpha_2 = 0$ and $\beta_3=0$.

Moreover using the expression for $P$ and the fact that  $JG(E_1,E_2)=\tfrac{\sqrt{3}}{3} E_3$ we have that $\ne_{E_1}E_1=\ne_{E_2}E_2=\ne_{E_3}E_3=0$ and
\begin{align*}
&\ne_{E_1}E_2 =0,~\ne_{E_1} E_3=0,\\
&\ne_{E_2}E_1 = - \tfrac{\sqrt{3}}{2} E_3=-\tfrac{\sqrt{3}}{2} (p \alpha_3,0),\\
& \ne_{E_2} E_3=\tfrac{\sqrt{3}}{4}(E_1-QE_1)= \tfrac{\sqrt{3}}{2} (p\alpha_1,0),\\
&\ne_{E_3}E_2 = -\tfrac{\sqrt{3}}{4}(E_1+QE_1)= -\tfrac{\sqrt{3}}{2} (0,q \alpha_1),\\
& \ne_{E_3}E_1 =  \tfrac{\sqrt{3}}{2} E_2=\tfrac{\sqrt{3}}{2} (0,q \beta_2).
\end{align*}

From the relation between the nearly K\"ahler metric and the  usual Euclidean product metric  (see \eqref{eq:metric}), we  have that  $E_1,E_2,E_3$ are orthogonal with respect to the induced Euclidean product metric and that their lengths are given by
$$\langle  E_1,E_1\rangle   =\tfrac 32, \qquad  \langle  E_2,E_2\rangle   =\langle  E_3,E_3\rangle   =\tfrac 34.$$
This in turn implies that $\alpha_1,\beta_2,\alpha_3$ are mutually orthogonal imaginary quaternions and
$$\vert \alpha_1 \vert^2 =\tfrac 34,  \qquad   \vert \beta_2 \vert^2=  \vert \alpha_3 \vert^2=\tfrac 34.$$

On the other hand, from
$$D_{E_j} D_{E_i} f = (p \alpha_j \alpha_i + p E_j(\alpha_i),
q \beta_j \beta_i + q E_j(\beta_i)),$$
it follows that
$$\ne_{E_j}E_i = (p (\alpha_j \times \alpha_i+E_j(\alpha_i)), q (\beta_j \times \beta_i +E_j(\beta_i)).$$
Hence substituting $\alpha_2=0$, $\beta_1=\alpha_1$ and $\beta_3 =0$ it follows
that
$$\beta_2 \times \alpha_1 = \tfrac{\sqrt{3}}{2} \alpha_3,$$
as well as
\begin{alignat*}{3}
& E_2(\alpha_1)=-\tfrac{\sqrt{3}}{2} \alpha_3, \qquad &&E_3(\alpha_1)=\tfrac{\sqrt{3}}{2} \beta_2,
 \qquad &&E_1(\alpha_1) = 0,\\
& E_2(\beta_2)=0, \qquad &&E_3(\beta_2)=-\tfrac{\sqrt{3}}{2} \alpha_1,
 \qquad &&E_1(\beta_2) = \tfrac{\sqrt{3}}{2} \alpha_3,\\
 & E_2(\alpha_3)=\tfrac{\sqrt{3}}{2} \alpha_1, \qquad &&E_3(\alpha_3)=0,
 \qquad &&E_1(\alpha_3) =-\tfrac{\sqrt{3}}{2} \beta_2.
\end{alignat*}
In terms of the standard vector fields $X_1,X_2,X_3$ this gives
\begin{alignat*}{3}
& X_1(\alpha_1)=2 \alpha_3, \qquad &&X_2(\alpha_1)=- 2\beta_2,
 \qquad &&X_3(\alpha_1) = 0,\\
& X_1(\beta_2)=0, \qquad &&X_2(\beta_2)=2 \alpha_1,
 \qquad &&X_3(\beta_2) = -2 \alpha_3,\\
 & X_1(\alpha_3)=-2 \alpha_1, \qquad &&X_2(\alpha_3)=0,
 \qquad &&X_3(\alpha_3) = 2 \beta_2.
\end{alignat*}

We can choose a rotation (unitary quaternion $h$) such that
\begin{align*}
 \beta_2(1) =\tfrac{\sqrt{3}}{2} h\mathbf{i}h^{-1},~
 \alpha_3(1) = \tfrac{\sqrt{3}}{2} h\mathbf{j}h^{-1},~
 \alpha_1(1) =-\tfrac{\sqrt{3}}{2} h\mathbf{k}h^{-1},
\end{align*}
and we can pick the initial conditions such that $f(1)= (h\mathbf{i}h^{-1},h\mathbf{j}h^{-1})$.
As the differential equations for $\alpha_i$, $\beta_i$, $p$ and $q$ are linear systems of differential equations with fixed initial conditions we can apply a standard uniqueness theorem. It is therefore  sufficient to give a solution which satisfies the above system with the given initial conditions.

We  have that
\begin{align*}
 \beta_2(u) =\tfrac{\sqrt{3}}{2} hu\mathbf{i}u^{-1}h^{-1},~
 \alpha_3(u) = \tfrac{\sqrt{3}}{2} hu\mathbf{j}u^{-1}h^{-1},~
\alpha_1(u) =-\tfrac{\sqrt{3}}{2} hu\mathbf{k}u^{-1}h^{-1},
\end{align*}
satisfy
$X_1(\beta_2)=X_2(\alpha_3)=X_3(\alpha_1)=0$ and
\begin{align*}
&X_1(\alpha_1)=-2\tfrac{\sqrt{3}}{2} hu\mathbf{i}\mathbf{k}u^{-1}h^{-1}=2\tfrac{\sqrt{3}}{2} hu\mathbf{j}u^{-1}h^{-1}=2 \alpha_3,\\
&X_2(\alpha_1)=-2\tfrac{\sqrt{3}}{2} hu\mathbf{j}\mathbf{k}u^{-1}h^{-1}=-2\tfrac{\sqrt{3}}{2} hu\mathbf{i}u^{-1}h^{-1}=-2 \beta_2,\\
&X_3(\alpha_3)=2\tfrac{\sqrt{3}}{2} hu(-\mathbf{k})\mathbf{j}u^{-1}h^{-1}=2\tfrac{\sqrt{3}}{2} hu\mathbf{i}u^{-1}h^{-1}=2 \beta_2,\\
&X_3(\beta_2)=2\tfrac{\sqrt{3}}{2} hu(-\mathbf{k})\mathbf{i}u^{-1}h^{-1}=-2\tfrac{\sqrt{3}}{2} hu\mathbf{j}u^{-1}h^{-1}=-2 \alpha_3,\\
&X_1(\alpha_3)=2\tfrac{\sqrt{3}}{2} hu\mathbf{i}\mathbf{j}u^{-1}h^{-1}=2\tfrac{\sqrt{3}}{2} hu\mathbf{k}u^{-1}h^{-1}=-2 \alpha_1,\\
&X_2(\beta_2)=2\tfrac{\sqrt{3}}{2} hu\mathbf{j}\mathbf{i}u^{-1}h^{-1}=-2\tfrac{\sqrt{3}}{2} hu\mathbf{k}u^{-1}h^{-1}=2 \alpha_1.
\end{align*}
Next, if we take $p=h u \mathbf{i} u^{-1} h^{-1}$ and $q = hu\mathbf{j} u^{-1} h^{-1}$, we  have that
\begin{align*}
&D_{E_1} p= -\tfrac{\sqrt{3}}{4} D_{X_3} p=\tfrac{\sqrt{3}}{2}h u \mathbf{j} u^{-1} h^{-1} =h u \mathbf{i} u^{-1} h^{-1}(-\tfrac{\sqrt{3}}{2} hu\mathbf{k}u^{-1}h^{-1})=p \alpha_1,\\
&D_{E_2} p=  -\tfrac{\sqrt{3}}{4} D_{X_1} p =0 =p \alpha_2,\\
&D_{E_3} p=  -\tfrac{\sqrt{3}}{4}  D_{X_2} p = \tfrac{\sqrt{3}}{2}h u \mathbf{k} u^{-1} h^{-1} =h u \mathbf{i} u^{-1} h^{-1}(\tfrac{\sqrt{3}}{2} hu\mathbf{j}u^{-1}h^{-1})=p \alpha_3,\\
&D_{E_1} q= -\tfrac{\sqrt{3}}{4} D_{X_3} q= \tfrac{\sqrt{3}}{2}h u (-\mathbf{i}) u^{-1} h^{-1} =h u \mathbf{j} u^{-1} h^{-1}(-\tfrac{\sqrt{3}}{2} hu\mathbf{k}u^{-1}h^{-1})=q \beta_1, \\
&D_{E_2} q=  -\tfrac{\sqrt{3}}{4} D_{X_1} q =-\tfrac{\sqrt{3}}{2}h u \mathbf{k} u^{-1} h^{-1} =h u \mathbf{j} u^{-1} h^{-1}(\tfrac{\sqrt{3}}{2} hu\mathbf{i}u^{-1}h^{-1})=q \beta_2,\\
&D_{E_3} q=  -\tfrac{\sqrt{3}}{4} D_{X_2} q = 0=q \beta_3.
\end{align*}
After applying  an isometry of  the nearly K\"ahler~$\nks$, we  completes the proof of the theorem.
\end{proof}

\textbf{Case~1b:} $\mathbf{h_{12}^3=-\tfrac{1}{2}}$.  In this case, all connection coefficients are zero and  the submanifold $M$  is flat.
In that case we have
\begin{theorem}\label{thmeg18} Let $M$ be a Lagrangian submanifold of  the nearly K\"ahler~$\nks$. Assume that there exists a local orthonormal frame as in \textbf{Case~1b}.
Then $M$ is locally congruent with
the immersion $f: \mathbb R^3\to \nks: (u,v,w)\mapsto (p(u,w),q(u,v))$, where $p$ and $q$ are constant mean curvature tori in $\mathbb{S}^3$ given in Example \ref{ex:l8}.
\end{theorem}
\begin{proof}
We know that all connection coefficients vanish
and that the only non vanishing component of the second fundamental form is $g(h(E_1,E_2), JE_3)= -\tfrac{1}{2}$. Moreover, from the beginning of the
 discussion about \textbf{Case 1}, we know that the angle functions are given by $(2 \theta_1,2 \theta_2,2 \theta_3)=(0,\tfrac{2\pi}{3}, \tfrac{4\pi}{3})$.
 So we can find a local basis such that $\sqrt{3} JG(E_1,E_2)=E_3$ and
\begin{align*}
PE_1=E_1,~
PE_2=-\tfrac 12 E_2 +\tfrac{\sqrt{3}}{2} JE_2,~
PE_3=-\tfrac 12 E_3 -\tfrac{\sqrt{3}}{2} JE_3.
\end{align*}
From this it follows that
\begin{align*}
QE_1=-\sqrt{3} JE_1,~
QE_2= E_2,~
QE_3= -E_3.
\end{align*}

As the connection coefficients vanish we may identify $E_1,E_2,E_3$ with coordinate vector fields. As before we write the $f= (p,q)$ and we denote the coordinates by $u,v,w$. Therefore, we have $E_1=f_u,E_2=f_v,E_3=f_w$. It immediately follows from the above expression of $Q$ that $p$  does not depend on $v$, $q$ does not depend on $w$ (i.e., $p_v=q_w=0$) and that $p^{-1} p_u= q^{-1} q_u$.

Moreover using the above expression for $P$ and the fact that  $JG(f_u,f_v)=\tfrac{\sqrt{3}}{3} f_w$, we have that $\ne_{f_u}f_u=\ne_{f_v}f_v=\ne_{f_w}f_w=0$ and
\begin{align*}
&\ne_{f_u}f_v=\ne_{f_v}f_u =-\tfrac 12 Jf_w -\tfrac{\sqrt{3}}{4}f_w -\tfrac 14 Jf_w=-\tfrac 34 Jf_w -\tfrac{\sqrt{3}}{4} f_w,\\
&\ne_{f_u} f_w=\ne_{f_w}f_u =-\tfrac 12 Jf_v
 +\tfrac{\sqrt{3}}{4}f_v -\tfrac 14 Jf_v=-\tfrac 34 Jf_v +\tfrac{\sqrt{3}}{4}
 f_v,\\
&\ne_{f_v}f_w =\ne_{f_w}{f_v} =-\tfrac 12 Jf_u +\tfrac 12 Jf_u=0.
\end{align*}

From the relation between the nearly K\"ahler metric and the  usual Euclidean product metric  (see \eqref{eq:metric}), we  have that  $f_u,f_v,f_w$ are also orthogonal with respect to the induced Euclidean product metric and that their lengths are given by
$$\langle  f_u,f_u\rangle   =\tfrac 32, \qquad  \langle  f_v,f_v\rangle   =\langle  f_w,f_w\rangle   =\tfrac 34.$$
Moreover from \eqref{eq:Q} and \eqref{eq:metric}, we have  $\langle  X,QY\rangle   =-\tfrac{\sqrt{3}}{2} g(X,JPY\rangle   $, hence
\begin{align*}
&\langle  f_u,Qf_v\rangle   =\langle  f_v,Qf_w\rangle   =\langle  f_u,Qf_w\rangle   =0,\\
&\langle  f_u,Qf_u\rangle   =0,~\langle  f_v,Qf_v\rangle   =\tfrac 34,~\langle  f_w,Qf_w\rangle   =-\tfrac 34.
\end{align*}

As
$$D_X Y = \ne_{X} Y -\tfrac 12 \langle  X,Y\rangle    f -\tfrac 12 \langle  X,QY\rangle    Qf,$$
where $D$ denotes the usual covariant derivative on $\mathbb H^2 = \mathbb R^8$, we deduce by combining the above equations that the immersion $f$ is determined by the following system of partial differential equations:
\begin{align*}
&f_{uu}=-\tfrac 34 f,~f_{vv}= -\tfrac 38 f - \tfrac 38 Qf,~f_{ww}= -\tfrac 38 f + \tfrac 38 Qf,\\
&f_{vw}=0,~f_{uv}=-\tfrac 34 Jf_w -\tfrac{\sqrt{3}}{4} f_w,~f_{uw}=-\tfrac 34 Jf_v +\tfrac{\sqrt{3}}{4} f_v.
\end{align*}

In terms of the components $p$ and $q$ this reduces to
\begin{equation}\label{puw}
p_{uu}=-\tfrac 34 p,\qquad \qquad  p_{ww}=-\tfrac 34 p, \qquad \qquad
p_{uw}=-\tfrac{\sqrt{3}}{2} pq^{-1} q_v.
\end{equation}
and
\begin{equation}\label{quv}
q_{uu}=-\tfrac 34 q, \qquad \qquad
q_{vv}=-\tfrac 34 q, \qquad \qquad
q_{uv}=\tfrac{\sqrt{3}}{2} qp^{-1} p_w.
\end{equation}
In order to simplify expressions, in the following we will write $\tilde u = \tfrac{\sqrt{3}}{2} u$,  $\tilde v = \tfrac{\sqrt{3}}{2} v$ and $\tilde w =\tfrac{\sqrt{3}}{2} w$.

We first look at the system of differential equations for $p$ (see \eqref{puw}). Solving the first two equations in \eqref{puw}, it follows that we can write
$$p=A_1 \cos(\tilde u)\cos (\tilde w) +A_2 \cos(\tilde u) \sin(\tilde w)+A_3 \sin(\tilde u)\cos(\tilde w) +A_4 \sin(\tilde u)\sin(\tilde w)),$$
where in order to simplify expressions we have written $\tilde u = \tfrac{\sqrt{3}}{2} u$,  $\tilde v = \tfrac{\sqrt{3}}{2} v$ and $\tilde w =\tfrac{\sqrt{3}}{2} w$.
Using now the fact that $\langle  p,p\rangle   =1$ together with $\langle  p_u,p_w\rangle   =0$ (as $f_u$ and $f_w$ are mutually orthogonal with respect to the induced Euclidean product metric) if follows that by applying an isometry of $SO(4)$ we may write $A_1=(1,0,0,0)$, $A_2=(0,1,0,0)$, $A_3=(0,0,1,0)$ and $A_4= (0,0,0,\epsilon_1)$, where $\epsilon_1= \pm 1$.

A similar argument is of course valid for the second map $q$.
Also it is well known that ~$\nks= SU(2)\times SU(2)$ is the double cover of~$SO(4)$, so any rotation~$R\in SO(4)$
can be written as~$R(x) = \alpha x \beta$, where~$\alpha, \beta \in \mathbb{S}^3$.
Therefore, applying an isometry of the nearly K\"ahler $\nks$, we can write that
\begin{align*}
&p=(\cos(\tilde u)\cos (\tilde w), \cos(\tilde u) \sin(\tilde w), \sin(\tilde u)\cos(\tilde w),\epsilon_1\sin(\tilde u)\sin(\tilde w)),\\
&q=(\cos(\tilde u)\cos (\tilde v), \cos(\tilde u) \sin(\tilde v), \sin(\tilde u)\cos(\tilde v),\epsilon_2\sin(\tilde u)\sin(\tilde w))d,
\end{align*}
where $\epsilon_i=\pm 1$ and $d=(d_1,d_2,d_3,d_4)$ is a unitary quaternion. Note that taking $d$ or $-d$ gives up to an isometry the same example.
Looking now at $p^{-1}p_{uw}+\tfrac{\sqrt{3}}{2} q^{-1} q_v=0$ (see \eqref{puw}) it immediately follows that $\epsilon_1=\epsilon_2=1$.

Moreover we get from $p_{uw}=-\tfrac{\sqrt{3}}{2} pq^{-1} q_v,~q_{uv}=\tfrac{\sqrt{3}}{2} qp^{-1} p_w$ (see \eqref{puw} and \eqref{quv})
 that the unit quaternion $d$ has to satisfy:
\begin{align*}
&d_1^2+d_2^2+d_3^2+d_4^2=1,\\
&d_1 d_2 +d_3 d_4=d_1 d_4-d_2 d_3=0,\\
&d_1^2+d_2^2-d_3^2-d_4^2=d_1^2-d_2^2-d_3^2+d_4^2=0,\\
&-1-2 d_1 d_3+2 d_2 d_4 =1+2 d_1 d_3+2 d_2 d_4=0.
\end{align*}
This reduces to
\begin{align*}
&d_3^2=d_1^2,~d_4^2=d_2^2,~d_1^2+d_2^2=\tfrac 12,~d_1 d_3=-\tfrac 12,~d_4 d_2=0.
\end{align*}
This system has  solutions $d=(\tfrac{1}{\sqrt{2}},0,-\tfrac{1}{\sqrt{2}},0)$ and
$d=(-\tfrac{1}{\sqrt{2}},0,\tfrac{1}{\sqrt{2}},0)$.
This completes the proof of the theorem.
\end{proof}

\textbf{Case~2:} $\mathbf{h_{12}^3=0.}$ In this case, recall that
$\bigl(\lambda_i(\lambda_j-\lambda_k)+\mu_i(\mu_j-\mu_k)\bigr)=2 \sin(\theta_j-\theta_k) \sin(2 \theta_i-\theta_j-\theta_k).$ Therefore, the general solution of~\eqref{eq:lm1} is
\begin{equation}\label{case2hijk}
\begin{aligned}
 h_{ii}^j &= -2\alpha_j \sin(\theta_j-\theta_k) \sin(2 \theta_i-\theta_j-\theta_k), \\
 h_{kk}^j &=\phantom{-} 2\alpha_j  \sin(\theta_i-\theta_j) \sin(2 \theta_k-\theta_i-\theta_j),
\end{aligned}
\end{equation}
where here and throughout the remainder of case, $(ijk)$ \textbf{denotes a positive permutations of} $(123)$ and $\alpha_1$,~$\alpha_2$ and~$\alpha_3$ are some real functions.
The components~$h_{ii}^i$ can be calculated using the minimality of~$M$.

As before we may assume that $M$ is not totally geodesic. Then from Lemma \ref{lem:equalangles} we may assume that all the angle functions are different (modulo~$\pi$). Hence, $\sin{(\theta_i-\theta_j)}\neq 0,~\forall~i\neq j$. We will show by contradiction that \textbf{Case 2} cannot occur.

By the second equation in~Lemma~\ref{lem:sff} and \eqref{case2hijk}, one then can express the~$\omega_{ij}^k$
in terms of the~$h_{ij}^k$ (since $\sin{(\theta_j-\theta_k)}\neq0,~\forall~j\neq k$). This gives us for  all positive permutations~$(ijk)$ of~$(123)$ that
\begin{align*}
&\omega_{ii}^j = 2 \alpha_j \cot(\theta_j-\theta_i) \sin(\theta_k-\theta_j) \sin(2 \theta_i-\theta_j-\theta_k),\\
&\omega_{kk}^j =2 \alpha_j \cot(\theta_j-\theta_k)\sin(\theta_i-\theta_j) \sin(2 \theta_k-\theta_i-\theta_j),\\
&\omega_{ij}^k = -\omega_{ik}^j = \tfrac{\sqrt{3}}{6}.
\end{align*}
Using Lemma \ref{lem:sff} and \eqref{case2hijk}, the differential equations for the angles become
\begin{equation}\label{dtheta}
\begin{aligned}
E_j(\theta_i)&= 2\alpha_j \sin(\theta_j-\theta_k) \sin(2 \theta_i-\theta_j-\theta_k),\\
E_j(\theta_k)&= -2\alpha_j  \sin(\theta_i-\theta_j) \sin(2 \theta_k-\theta_i-\theta_j),
\end{aligned}
\end{equation}
and
\begin{equation}\label{dtheta2}
\begin{aligned}
E_1(\theta_1)&=-\alpha_1 (\cos (2 (\theta_1-\theta_2))+\cos (2 (\theta_1-\theta_3))-2 \cos (2 (\theta_2-\theta_3))),\\
E_2(\theta_2)&=-\alpha_2 (\cos (2 (\theta_2-\theta_3))+\cos (2 (\theta_2-\theta_1))-2 \cos (2 (\theta_3-\theta_1))),\\
E_3(\theta_3)&=-\alpha_3 (\cos (2 (\theta_3-\theta_1))+\cos (2 (\theta_3-\theta_2))-2 \cos (2 (\theta_1-\theta_2))).
\end{aligned}
\end{equation}

First, we  deal with the case that $ \cos(\theta_i -\theta_j) \neq 0\neq\sin(2 \theta_i-\theta_j-\theta_k),~\forall ~i,j,k ~\text{distinct}$. In that case, we
find using the above expressions for~$\omega_{ij}^k$ together with  the Gauss equation \eqref{gauss} for  $R(E_i,E_j)E_k$ that
\begin{align*}
  & E_i(\alpha_j)= -\frac{1}{6}\csc(\theta_i-\theta_j)\sin(\theta_i-\theta_k)\csc(\theta_i+\theta_j-2\theta_k)\times\\
   &\phantom{\frac{1}{6}}\Bigl[
      6 \alpha_i \alpha_j \sin(\theta_i+\theta_j-2 \theta_k) (-7 \sin(\theta_j-\theta_k) +2 \sin(2 \theta_i -\theta_j-\theta_k) +\sin(2 \theta_i-3 \theta_j+\theta_k)) \\
      &\phantom{\frac{1}{6}} +2 \sqrt{3} \alpha_k \sin(\theta_i - 2 \theta_j + \theta_k)
     \Bigr]
\end{align*}
and
\begin{align*}
  & E_j(\alpha_i)= \frac{1}{6}\csc(\theta_i-\theta_j)\sin(\theta_j-\theta_k)\csc(\theta_i+\theta_j-2\theta_k)\times\\
   &\phantom{\frac{1}{6}}\Bigl[
      6 \alpha_i \alpha_j \sin(\theta_i+\theta_j-2 \theta_k) (-7 \sin(\theta_i-\theta_k) +2 \sin(2 \theta_j -\theta_i-\theta_k) +\sin(2 \theta_j-3 \theta_i+\theta_k)) \\
      &\phantom{\frac{1}{6}} + 2 \sqrt{3} \alpha_2 \sin(2 \theta_i -  \theta_j - \theta_k)
     \Bigr].
\end{align*}
Substituting these derivatives into the compatibility conditions,
$$E_i(E_j(\theta_i))-E_j(E_i(\theta_i))=(\nabla_{E_i}E_j-\nabla_{E_j}E_i)(\theta_i),$$
gives the following three equations relating the functions $\alpha_1$, $\alpha_2$ and $\alpha_3$:
\begin{equation}\label{alphai}
 \alpha_i \bigl(\sin(2 \theta_k -  \theta_i - \theta_j)\sin( \theta_k +  \theta_i - 2\theta_j)\bigr) \\
 = 8 \sqrt{3} \alpha_j \alpha_k \sin^3(\theta_k - \theta_j) \sin^2(\theta_k - 2 \theta_i + \theta_j),
\end{equation}
for every positive permutation~$(ijk)$ of~$(123)$. So we have equations of the form~$x_i \alpha_i = \alpha_j \alpha_k$,
where
\begin{equation*}
x_i = \frac{\sin(2 \theta_k -  \theta_i - \theta_j)\sin( \theta_k +  \theta_i - 2\theta_j)}{8 \sqrt{3}  \sin^3(\theta_k - \theta_j) \sin^2(\theta_k - 2 \theta_i + \theta_j)} \not = 0.
\end{equation*}
As the lagrangian submanifold $M$ is not totally geodesic, from  \eqref{case2hijk} we know that not all the $\alpha_i$ can vanish at the same time. Therefore it follows from the above system of equations \eqref{alphai} that
\begin{equation*}
\alpha_i^2 = x_j x_k= -\tfrac{\sin^2(2 \theta_i -  \theta_j - \theta_j)}{192 \sin ^3(\theta_i-\theta_j) \sin ^3(\theta_i-\theta _k) \sin
   (\theta_i+\theta_j-2\theta_k) \sin (\theta_i-2 \theta_j+\theta_k)}.
\end{equation*}
Using the Gauss equation \eqref{gauss} to calculate the sectional curvature~$K$
of the plane spanned by $E_i$ and $E_j$, we have
that~$K= \tfrac{5}{12}+ \tfrac{1}{3} \cos(2(\theta_i - \theta_j))$ for all~$i\neq j$.
As the sectional curvature is constant, this implies that
$$\cos(2(\theta_i - \theta_j))=\cos(2(\theta_k - \theta_j)),$$
which  means that all the angles are constant, hence by Lemma \ref{lem:sff} and the assumption of \textbf{Case 2} that $h_{12}^3=0$,  the submanifold $M$  is totally geodesic. So we get a contradiction with
the assumption that $M$ is not totally geodesic.

Next, we deal with the case that there exist some $i,j,k $ which are distinct such that $\sin(2\theta_i-\theta_j-\theta_k) = 0$.
As the sum of the angles is a multiple of $\pi$ and all the angles are determined up to a multiple of $\pi$, in this case it is sufficient to consider the case that $\theta_1= \frac{b \pi}{3}$ and $\theta_3 = a\pi-\theta_1-\theta_2=a\pi-\frac{b\pi}{3}-\theta_2$, where $a$ and $b$ are some constant integers and $\theta_2$ is not constant.
As $\theta_2+\theta_3=a\pi-\frac{b\pi}{3}=\text{constant}$, $\theta_2$ is not constant, using \eqref{dtheta},  we get from
$$0=E_1(\theta_2+\theta_3)=2\alpha_1[1-2 \cos ^2{(a\pi-\frac{b\pi  }{3}-2 \theta_2)}+(-1)^{a-b} \cos{(a\pi-\frac{b\pi}{3}-2\theta_2)}]$$
that $\alpha_1$ has to vanish.
Then from the Gauss equation \eqref{gauss} we obtain that the sectional curvature $K$  of the plane spanned by $E_1$ and $E_2$ is given by
$$K=\tfrac{5}{12} + \tfrac{1}{3} \cos(2 \theta_2-\frac{2b\pi}{3}).$$
As $M$ has constant sectional curvature this implies that $\theta_2$, and therefore all angle functions, are constant. Hence by Lemma \ref{lem:sff} and the assumption of \textbf{Case 2} that $h_{12}^3=0$,  the submanifold $M$  is totally geodesic. So we get a contradiction with
the assumption that $M$ is not totally geodesic.

Finally, we deal with the case that there exist some $i,j$  such that $\cos(\theta_i-\theta_j)=0$. As the sum of the angles is a multiple of $\pi$ and all the angles are only determined up to a multiple of $\pi$, in this case it is sufficient to consider the case that
\begin{align*}
\theta_2 = \theta_1 -\tfrac{b\pi}{2},~\theta_3=a\pi-\theta_1-\theta_2= a\pi+\tfrac{b\pi}{2}-2\theta_1,
\end{align*}
 where $a$ is a constant integer,  $b$ is an odd constant integer and $\theta_1$ is not constant.
 As $\theta_1-\theta_2=\frac{b\pi}{2}=\text{constant}$, $\theta_1$ is not constant, using \eqref{dtheta}-\eqref{dtheta2},  we get from
 \begin{equation*}
 \begin{aligned}
&0=E_1(\theta_1-\theta_2)=-\alpha_1\big(2\cos{(b\pi)}-3\cos{(6\theta_1)}+\cos{(b\pi-6\theta_1})\big),\\
&0=E_2(\theta_1-\theta_2)=\alpha_2\big(2\cos{(b\pi)}+\cos{(6\theta_1)}-3\cos{(b\pi-6\theta_1)}\big),\\
&0=E_3(\theta_1-\theta_2)=2\alpha_3\sin{(\frac{b\pi}{2})}\sin{(\frac{3b\pi}{2}-6\theta_1)},
\end{aligned}
\end{equation*}
that $\alpha_1=\alpha_2=\alpha_3=0$. Hence by Lemma \ref{lem:sff} and the assumption of \textbf{Case 2} that $h_{12}^3=0$,  the submanifold $M$  is totally geodesic. So we get a contradiction with
the assumption that $M$ is not totally geodesic.
\medskip

Therefore, we have proved that \textbf{Case 2} cannot occur.
\medskip

\noindent \textbf{Proof of Theorem \ref{main}~:}
Assume that $M$ is a Lagrangian submanifold of constant  sectional curvature in the nearly K\"ahler~$\nks$.
First, we consider the case that $M$ is totally geodesic, then applying Theorem \ref{thmtg} obtained by Zhang-Hu-Dioos-Vrancken-Wang, we get that
$M$ is locally congruent with  one of the following
immersions:

(1)  $f\colon \mathbb{S}^3 \to\nks: u\mapsto (u,1)$, which is Example \ref{ex:l1},

(2) $f\colon \mathbb{S}^3 \to\nks: u\mapsto (1,u)$, which is Example \ref{ex:l2},

(3) $f\colon \mathbb{S}^3\to\nks: u\mapsto (u,u)$, which is Example \ref{ex:l3}.

 Second, we consider the case that $M$ is not totally geodesic. They applying our discussions for Case 1a (see Theorem \ref{thmeg16}), Case 1b (see Theorem \ref{thmeg18}) and Case 2 (we have proved that this case cannot occur), we obtain that $M$ is locally congruent with  one of the following
immersions:

(4)  $f\colon \mathbb{S}^3\to \nks: u\mapsto (u\mathbf{i}u^{-1},u\mathbf{j}u^{-1})$, which is Example \ref{ex:l6},

(5) $f: \mathbb R^3\to \nks: (u,v,w)\mapsto (p(u,w),q(u,v))$, where $p$ and $q$ are constant mean curvature tori in $\mathbb{S}^3$ given in Example \ref{ex:l8}.

 This complete the proof of Theorem \ref{main}.
\qed

\end{document}